\documentclass[11pt, letter, reqno]{amsart}
\usepackage{t1enc}
\usepackage[latin1]{inputenc}
\usepackage[english]{babel}
\usepackage{amsmath,amsthm}
\usepackage{amsfonts}
\usepackage{latexsym}
\usepackage[dvips]{graphicx}
\usepackage{float}
\usepackage[natural]{xcolor}
\usepackage{algorithm}
\usepackage{algorithmic}
\usepackage{enumerate,enumitem}
\usepackage{multirow}
\usepackage{xcolor}
\usepackage[colorlinks,linkcolor=blue]{hyperref}
\usepackage{lineno}
\usepackage{setspace}
\usepackage{fullpage}
\usepackage[title]{appendix}
\usepackage{todonotes}
\usepackage{amssymb}

\usepackage{listings}
\newcommand{\eps}{\varepsilon}
\newcommand{\<}{\subseteq}

\newcommand{\De}{\Delta}
\newcommand{\de}{\delta}
\newcommand{\be}{\beta}

\newcommand{\ga}{\gamma}
\newcommand{\la}{\lambda}

\def \F{\mathcal{F}}

\def \N{\mathbb{N}}

\usepackage{bbm}
\newcounter{propcounter}
\usepackage{enumitem}
\theoremstyle{plain}
\newtheorem{thm}{Theorem}[section]
\newtheorem{theorem}[thm]{Theorem}
\newtheorem{conjecture}[thm]{Conjecture}
\newtheorem{lemma}[thm]{Lemma}
\newtheorem{corollary}[thm]{Corollary}
\newtheorem{proposition}[thm]{Proposition}

\theoremstyle{definition}

\newtheorem{question}[thm]{Question}

\newtheorem{remark}[thm]{Remark}

\newtheorem{definition}[thm]{Definition}
\newtheorem{claim}[thm]{Claim}
\newtheorem{fact}[thm]{Fact}

\newtheorem{example}[thm]{Example}

\newtheorem{defn-thm}[thm]{Definition-Theorem}
\numberwithin{equation}{section}

\newcommand{\btheorem}{\begin{theorem}}
\newcommand{\etheorem}{\end{theorem}}
\newcommand{\bconjecture}{\begin{conjecture}}
\newcommand{\econjecture}{\end{conjecture}}
\newcommand{\bproposition}{\begin{proposition}}
\newcommand{\eproposition}{\end{proposition}}
\newcommand{\bdefinition}{\begin{definition}}
\newcommand{\edefinition}{\end{definition}}
\newcommand{\bcorollary}{\begin{corollary}}
\newcommand{\ecorollary}{\end{corollary}}
\newcommand{\pr}{\begin{proof}}
\newcommand{\oof}{\end{proof}}
\newcommand{\bclaim}{\begin{claim}}
\newcommand{\eclaim}{\end{claim}}
\newcommand{\bquestion}{\begin{question}}
\newcommand{\equestion}{\end{question}}
\newcommand{\bfact}{\begin{fact}}
\newcommand{\efact}{\end{fact}}
\newcommand{\bremark}{\begin{remark}}
\newcommand{\eremark}{\end{remark}}
\newcommand{\eexample}{\end{example}}
\newcommand{\bexample}{\begin{example}}
\newcommand{\ma}{\end{lemma}}
\newcommand{\lem}{\begin{lemma}}

\begin{document}
%\linenumbers

%\begin{titlepage}

%\title{Tree-universality in pseudorandom graphs}
\title{Spanning trees in sparse expanders}

\author{Jie Han}
\address{School of Mathematics and Statistics, Beijing Institute of Technology, Beijing, China. Email: {\tt han.jie@bit.edu.cn.}}

\author{Donglei Yang}
\address{Data Science Institute, Shandong University, Shandong, China. Email: {\tt dlyang@sdu.edu.cn}. }
%\thanks{The second author is partially supported by the China Postdoctoral Science Foundation (2021T140413), Natural Science Foundation of China (12101365) and Natural Science Foundation of Shandong Province (ZR2021QA029).}

\begin{abstract}
Given integers $n\ge \Delta\ge 2$, let $\mathcal{T}(n, \Delta)$ be the collection of all $n$-vertex trees with maximum degree at most $\Delta$.
A question of Alon, Krivelevich and Sudakov in 2007 asks for determining the best possible spectral gap condition forcing an $(n, d,\lambda)$-graph to be $\mathcal{T}(n, \Delta)$-universal, namely, it contains all members of $\mathcal{T}(n, \Delta)$ as a subgraph simultaneously.
In this paper we show that for sufficiently large integer $n$ and all $\Delta\in \mathbb{N}$, every
$(n, d,\lambda)$-graph with
\[
\lambda\le\frac{d}{2\Delta^{5\sqrt{\log n}}}
\]
is $\mathcal{T}(n, \Delta)$-universal.
As an immediate corollary, this implies that Alon's ingenious construction of triangle-free sparse expander is $\mathcal{T}(n, \Delta)$-universal, which provides an explicit construction of such graphs and thus solves a question of Johannsen, Krivelevich and Samotij.
Our main result is formulated under a much more general context, namely, the $(n,d)$-expanders.
More precisely, we show that there exist absolute constants $C,c>0$ such that the
following statement holds for sufficiently large integer $n$.
\begin{enumerate}
  \item For all $\Delta\in \mathbb{N}$, every
$(n, \Delta^{5\sqrt{\log n}})$-expander is $\mathcal{T}(n, \Delta)$-universal.
  \item For all $\Delta\in \mathbb{N}$ with $\Delta \le c\sqrt{n}$, every
$(n, C\Delta n^{1/2})$-expander is $\mathcal{T}(n, \Delta)$-universal.
\end{enumerate}
Both results significantly improve a result of Johannsen, Krivelevich and Samotij, and have further implications in locally sparse expanders and Maker-Breaker games that also improve previously known results drastically.
\end{abstract}
% and (2) is tight up to a factor of $\log n$ -- such a result does not hold if we replace $C\Delta n^{1/2}$ by $C\Delta n^{1/2}\log^{-1} n$.

%find the asymptotical sharp minimum degree threshold for $n$-vertex graphs $G$ with $\alpha_\ell(G)\in (n^{1-\gamma}, n^{1-o(1)})$ for any positive constant $\gamma<\frac{\ell-1}{\ell^2+2\ell}$.%$ depending only on $\ell$.
%The following question was proposed by Nenadov and Pehova and reiterated by Knierim and Su: Given integers $r>\ell\ge2$ and $n\in r\mathbb{N}$, is it true that every $n$-vertex graph $G$ with $\delta(G) \ge \max \{ \frac{1}{2},\frac{r - \ell}{r} \}n + o(n)$ and $\alpha_{\ell}(G) = o(n) $ contains a $K_{r}$-factor? We give an answer to this question as follows with a stronger bound on the independence number.

%Let $\lambda:=1/\lfloor\frac{r}{\ell}+1\rfloor, \omega(n)\rightarrow\infty$ slowly and $f(n):=n^{1-\omega(n)\log^{-\lambda}n}$. We show that for any $\mu >0$, if $n\in r\mathbb{N}$ is sufficiently large, then every $n$-vertex graph $G$ with $\delta(G) \ge \max\{\frac{r-\ell}{r}+\mu, \frac{1}{2-\varrho^*_{\ell}(r-1,f)} + \mu \}n $ and $\alpha_{\ell}(G) <  f(n)$  contains a $K_{r}$-factor, where $\varrho^*_{\ell}(r-1,f)$ is the degree-version Ramsey--Tur\'an density of $K_{r-1}$ under the $\ell$-independence number condition.

\maketitle
%\end{titlepage}

\section{Introduction}
A graph $G$ is called \emph{universal} for a family of graphs $\F$ for every $F\in \F$, $F$ is a subgraph of $G$.
There has been a rich body of research on explicit or randomized constructions of universal graphs~\cite{A2,A3,A4,B6,B12,C16,C17,C19,C20}.
In this paper, we focus on the case when $\F$ is a family of \emph{spanning trees} with bounded maximum degree.
The problem of existence of large trees in graphs and random
graphs has a long and profound history. %with most of the results being devoted to finding trees with bounded maximum degree.
%\bigskip

For integers $n, \Delta\in \mathbb N$, we define $\mathcal{T}(n, \Delta)$ as the family of all $n$-vertex trees with maximum degree at most $\Delta$.
%\textcolor[rgb]{0.00,0.00,1.00}{$\mathcal{T}(n, \Delta)$}: the class of trees with $n$ vertices and maximum degree at most $\Delta$.
The binomial random graphs $G(n,p)$ are $n$-vertex graphs where every pair of vertices is connected with probability $p$, independent of other pairs.
Regarding the containment of a single bounded degree tree in random graphs, Kahn~\cite{kahn16} conjectured that for any $\Delta>0$, there exists a constant $C:=C(\Delta)$ such that $C(\Delta)\log(n)/n$ is the threshold\footnote{Given a non-trivial monotone graph property $\mathcal P$, a function $q(n)$ is called a threshold function for $\mathcal P$ if when $p(n)/q(n)\to 0$ the probability that $G(n,p)$ satisfies $\mathcal P$ tends to 0, and when  $p(n)/q(n)\to \infty$ the probability that $G(n,p)$ satisfies $\mathcal P$ tends to 1.}
for a spanning tree with maximum degree $\Delta$ in $G(n,p)$.
In particular, this would imply that the threshold for constant maximum degree spanning trees is $\log(n)/n$.
The conjecture of Kahn is resolved by a breakthrough of Montgomery~\cite{Montgomery19}, who actually showed that at the conjectured threshold one actually gets the universality for the family $\mathcal{T}(n, \Delta)$.
%
%\begin{conjecture}[\cite{kahn16}]\label{conj1}
%For all $\Delta\in \mathbb{N}$, there should be some $C > 0$ such that for every tree $T \in \mathcal{T} (n, \Delta)$, $G(n, \frac{C\log n}{n})$ a.a.s contains a
%copy of $T$.
%\end{conjecture}

%\bigskip
%The first progress towards this was made by Alon--Krivelevich--Sudakov, where a graph is called \textcolor[rgb]{0.00,0.00,1.00}{universal} for a given graph class $\mathcal{F}$ (or, equivalently, \textcolor[rgb]{0.00,0.00,1.00}{$\mathcal{F}$-universal}) if it contains a copy of every graph in $\mathcal{F}$ as a subgraph.

%\begin{theorem}[\cite{Alon07}]
%For all $\eps\in (0, 1)$ and $\Delta\in \mathbb{N}$, there exists a
%constant $C$ such that $G(n, \frac{C}{n})$ is a.a.s $\mathcal{T}(n-\eps n, \Delta)$-universal.
%\end{theorem}
%
%\emph{`` In case of the path,
%this question is very well understood and it is known that for $p = O(\log n/n)$ the random graph $G(n, p)$ a.a.s contains a Hamiltonian path. We believe that a more general result should
%be true, i.e., such a random graph should already contain a.a.s every tree
%on $n$ vertices with maximum degree at most $\Delta$. Our methods are insufficient to attack this problem''}.

%Montgomery~\cite{Montgomery19} recently solved Conjecture~\ref{conj1} and also showed a more general universality result as follows.
\begin{theorem}\cite{Montgomery19}
\label{thm:RM19}
There exists $C=C(\Delta)$ such that $G(n, \frac{C\log n}{n})$ is a.a.s $\mathcal{T}(n, \Delta)$-universal.
\end{theorem}

Montgomery's result for single tree containment is further obtained recently by Frankston, Kahn, Narayanan and Park~\cite{FRANK21} and by Park and Pham~\cite{PP} in their resolutions of the fractional expectation-threshold conjecture of Talagrand~\cite{TALA10}, and the expectation-threshold conjecture itself due to Kahn and Kalai~\cite{KK07}.

These recent celebrated developments extend drastically our knowledge on \emph{thresholds}, the optimal probability for the emergence of subgraphs in random graphs.
It also raises naturally the question of constructions of sparse universal graphs for spanning trees and other natural graph classes.

\subsection{Universality in sparse graphs}
Another natural graph class is {$\mathcal H(n, \Delta)$}, the class of all $n$-vertex graphs with maximum degree at most $\Delta$.
For dense random graphs $G(n,p)$ where $p$ is a constant, the celebrated blow-up lemma of Koml\'os, S\'ark\"ozy and Szemer\'edi~\cite{KSS01} shows that $G(n,p)$ is universal for $\mathcal H(n, \Delta)$ when $\Delta$ is a constant.
Unlike the dense case, the study of universality in sparse random graphs is proven to be a challenging task.
Following some initial progress by Krivelevich~\cite{1997} and Kim~\cite{2003}, Johansson, Kahn and Vu~\cite{JKV} showed (among other things) that the threshold for the existence of a $K_{\Delta+1}$-factor is $n^{-\frac2{\Delta+1}}(\log n)^{1/\binom {\Delta+1}2}$.
For universality, Dellamonica, Kohayakawa, R\"{o}dl and Ruci\'{n}ski~\cite{D12} (for $\Delta\ge 3$) and Kim and Lee~\cite{KL14} (for $\Delta=2$) showed that $G(n, p)$ are $\mathcal H(n, \Delta)$-universal for $p=\tilde\Omega(n^{-1/\Delta})$.
The case $\Delta\ge 3$ has been subsequently improved by Ferber and Nenadov~\cite{FN18}, and the case $\Delta=2$, namely, for the 2-universality, the optimal threshold $p=C\left( \frac{\log n}{n^2}\right)^{1/3}$ was established in an excellent work of Ferber, Kronenberg and Luh~\cite{FER19}.

Now let us switch to another fruitful area, the study of sparse pseudorandom graphs.
One prominent class of such graphs are expander graphs.
Given a graph $G$ on $n$ vertices, let $\la_1\ge\la_2\ge \cdots\ge\la_n$ be the eigenvalues of its adjacency
matrix. Then $\la(G)=\max_{2\le i\le n} |\la_i|$ is called \emph{the second eigenvalue} of $G$.
An $(n,d,\la)$-\emph{graph} $G$ is a $d$-regular graph on $n$
vertices and the second eigenvalue of $G$ is at most $\la$ (its first eigenvalue is equal to $d$).
The well-known Expander Mixing Lemma gives a good estimate on the edge distribution of an $(n,d,\la)$-graph, which suggests that if $\la$ is much smaller than the degree $d$, then $G$ has strong expansion properties. Therefore, many typical embedding problems asks for the best possible condition on \emph{spectral gap} $\frac{\lambda}{d}$ to ensure certain structures.
For instance, a celebrated conjecture of Krivelevich and Sudakov~\cite{KS03} states that there exists $c>0$ such that for large enough $n$, every $(n, d,\lambda)$-graph with $\frac{\lambda}{d}<c$ is Hamiltonian.
Similar to sparse random graphs, embedding (and universality) problems appear to be very hard in $(n,d,\la)$-graphs.
The sparse blow-up lemma developed in~\cite{ALLEN19} shows that $(n,d,\la)$-graphs with $\lambda = o((d/n)^{\max\{4, (3\Delta+1)/2\}}n)$ are $\mathcal H(n, \Delta)$-universal.

%Back to the tree-universal problem,
In this paper, we  shall focus on the following problem on tree-universality proposed by Alon, Krivelevich and Sudakov~\cite{Alon07} in 2007.
\begin{question}\cite{Alon07}
Is it true that for every $\Delta\in \mathbb{N}$, there exists $c=c(\Delta)>0$ such that every
$(n, d,\lambda)$-graph with $\frac{\lambda}{d}<c$ is $\mathcal{T}(n, \Delta)$-universal?
\end{question}

Alon, Krivelevich and Sudakov~\cite{Alon07} obtained the following result on embedding almost spanning trees (subsequent improvement on the spectral gap was obtained by Balogh,
Csaba, Pei and Samotij~\cite{BCPS10}).
\begin{theorem}\cite{Alon07,BCPS10}
For all $\Delta\in \mathbb{N}$ and constant $0<\eps<1/2$, every
$(n, d,\lambda)$-graph with \[\frac{\lambda}{d}\le\frac{\eps}{\sqrt{8\De}}\] is $\mathcal{T}(n-\eps n, \Delta)$-universal.
\end{theorem}

Our first main result is as follows.

\begin{theorem}\label{th1}
For sufficiently large integer $n$ and all $\Delta\in \mathbb{N}$, every
$(n, d,\lambda)$-graph with
\[
\la\le \frac{d}{2\De^{5\sqrt{\log n}}}
\]
is $\mathcal{T}(n, \Delta)$-universal.
\end{theorem}

Note that when $\Delta = 2^{o(\sqrt{\log n})}$, it holds that $\De^{5\sqrt{\log n}} = n^{o(1)}$, namely, Theorem~\ref{th1} narrows the spectral gap which bypasses any polynomial function of $n$.
%We also remark that the proof of Theorem~\ref{th1} can be easily adapted to give a similar result for $(p, \beta)$-bijumbled graphs with a minimum degree condition.
We also remark that our proof also works under a slightly more general context, namely, the $(p, \beta)$-bijumbled graphs with a minimum degree condition.
Indeed, a graph $G=(V, E)$ is $(p, \beta)$-bijumbled if for every vertex sets $X, Y\subseteq V$, we have
\[
|e_G(X, Y) - p|X||Y|| \le \beta \sqrt{|X||Y|},
\]
where $e_G(X, Y):=|\{xy\in E(G) \mid (x, y)\in X\times Y\}|$ (so edges in $X\cap Y$ are counted twice).
It is well-known that $(n, d,\lambda)$-graphs are $(d/n, \lambda)$-bijumbled.

\begin{theorem}
\label{thm:1'}
For sufficiently large $n$ and all $\Delta\in \mathbb{N}$, every $n$-vertex
$(p, \beta)$-bijumbled with
\[
{\beta} \le \frac{pn}{4\De^{5\sqrt{\log n}}}
\]
and minimum degree at least $4\sqrt{\be pn}$
is $\mathcal{T}(n, \Delta)$-universal.
\end{theorem}

Note that the minimum degree condition we enforce in Theorem~\ref{thm:1'} is significantly weaker than that in Theorem~\ref{th1}, which considers $d$-regular (namely, $pn$-regular) graphs.

\subsection{Universality for $(n,d)$-expanders}

Our next result is on the universality problem in a broader class of graphs, namely, the $(n,d)$-expander graphs, defined as follows.
The notion of $(n, d)$-expander is first formulated by Johannsen, Krivelevich and Samotij, which is an adaptation of the expansion properties investigated by Hefetz, Krivelevich and Szab\'{o}~\cite{HKS09} for the Hamiltonicity of highly connected graphs.

\begin{definition} \cite{JKS12}\label{expander}
Given $n \in\N$ and $d>0$, an $n$-vertex graph $G$ is an $(n, d)$-\emph{expander}
if $G$ satisfies the following two conditions:
\begin{enumerate}[label=(\roman*)]
  \item $|N_G(X)| \ge d|X|$ for all $X \subseteq V(G)$ with $1 \le |X| < \frac{n}{2d}$;
  \item $e_G(X, Y) > 0$ for all disjoint $X, Y \subseteq V(G)$ with $|X| = |Y | \ge\frac{n}{2d}$.
\end{enumerate}
\end{definition}
Montgomery~\cite{Montgomery14} proposed a slightly stronger version of expansion property as follows.
\begin{definition}\label{expands}
For a graph $G$ and a set $W\< V(G)$, we say $G$ \emph{$d$-expands} into $W$ if
\begin{enumerate}[label=(\roman*)]
\item $|N_G(X, W)|\geq d|X|$ for all $X\<V(G)$ with $1\leq |X|<\lceil \frac{|W|}{2d}\rceil$, and,
\item $e_G(X,Y)>0$ for all disjoint $X,Y\<V(G)$ with $|X|=|Y|=\lceil\frac{|W|}{2d}\rceil$.
\end{enumerate}
\end{definition}
Clearly, if $G$ $d$-expands into $W$, then by definition $G[W]$ itself is a $(|W|,d)$-expander.
More often, we call a graph $G$ $m$-\emph{joined} if there is an edge in $G$
between every two disjoint sets of at least $m$ vertices.

%Tree-universality in $(n, d)$-expanders have been studied by

Johanssen, Krivelevich and Samotij~\cite{JKS12} were the first to show that expander graphs are tree-universal, and thus provide explicit constructions of sparse tree-universal graphs, e.g. by the celebrated construction of Ramanujan graphs by Lubotzky-Phillips-Sarnak~\cite{LPS88}.
Since expanders are locally sparse (e.g. with small clique number), this complements previous constructions of Bhatt, Chung, Leighton and Rosenberg~\cite{B12} that are locally dense, namely, their constructions contain a large number of cliques of size $\Omega(\Delta)$.
Now we present the results of Johanssen, Krivelevich and Samotij~\cite{JKS12} on spanning tree universality.

\begin{theorem}\cite{JKS12}\label{thm:JKS}
There exists an absolute constant $c>0$ such that the
following statement holds. For all $n,\Delta\in \mathbb{N}$ with $\Delta \le cn^{1/3}$, every
$(n, 7\Delta n^{2/3})$-expander is $\mathcal{T}(n, \Delta)$-universal.
\end{theorem}

%\blue{They pointed out that for $\De= n^{\Omega(1)}$, the smallest value
%$d$ necessary for every $(n, d)$-expander to be $\mathcal{T}(n, \Delta)$-universal grows faster than $\De^{1+\Omega(1)}$. In fact, they achieved this by constructing a very strong expander with relatively large radius (than that of the complete $(\De-1)$-ary tree).}
%\begin{theorem}\cite{JKS12}\label{lower}
%There exists an absolute constant $c>0$ such that the following statement holds. Let $n, r \in \N$ and let $\De := n^{
%1/(r+1)} + 2$ satisfy $ \De\ge c^{-1}\log n$. Then there exists an
%$(n, c\De^{1+1/r} \log^{-1} n)$-expander which is not $\mathcal{T}(n, \Delta)$-universal.
%\end{theorem}

They also proved the following result on the containment of almost spanning trees.

\begin{theorem}\cite{JKS12}\label{almost}
Let $n, \Delta\in \mathbb{N}$ and $d \ge 2\Delta$. Then every $(n, d)$-expander is $\mathcal{T}\left(n - 4\Delta \lceil\frac{n}{2d}\rceil, \Delta\right)$-universal.
\end{theorem}

Our second main result improves Theorem~\ref{thm:JKS}.

\begin{theorem}\label{th2}
There exist absolute constants $C,c>0$ such that the
following statement holds for sufficiently large integer $n$. For all $\Delta\in \mathbb{N}$ with $\Delta \le c\sqrt{n}$, every
$(n, C\Delta \sqrt{n})$-expander is $\mathcal{T}(n, \Delta)$-universal.
\end{theorem}

%In fact, by Theorem~\ref{lower} with $r=1$, when $\Delta$ is large, the term $C\Delta \sqrt{n}$ in Theorem~\ref{th2} is tight up to a factor of $\log n$, namely, it cannot be replaced by a term of form $O(\Delta \sqrt{n}\log^{-1}(n))$.

For smaller values of $\Delta$, we further improve the result as follows.
\begin{theorem}\label{th1+}
There exists $n_0\in \N$ such that for all integers $n,\De$ with $n\ge n_0$ and $\Delta\ge 2$, every
$(n, d)$-expander with \[d\ge\De^{5\sqrt{\log n}}\] is $\mathcal{T}(n, \Delta)$-universal.
\end{theorem}

%One crucial application of Theorem~\ref{th1+} is Theorem~\ref{th1} and here we need the following expansion property of $(n, d,\lambda)$-graphs, which can be immediately derived by applying Expander Mixing Lemma.
In fact, Theorem~\ref{th1+} implies Theorem~\ref{th1} and Theorem~\ref{thm:1'} by the following observations.

\begin{proposition}\label{f1}
Given $n,p,\be$ with $0<\beta\le\frac{pn}{400}$, every $n$-vertex $(p,\beta)$-bijumbled graph with minimum degree at least $4\sqrt{p\be n}$ is an $(n, d_1)$-expander for $d_1=\frac{pn}{4\beta}$. Moreover, every $(n, d,\lambda)$-graph with $\lambda<\frac{d}{8}$ is an $(n,\frac{d}{2\la})$-expander.
\end{proposition}
We include a short proof of Proposition~\ref{f1} in Appendix~\ref{APP2}.
At last, we remark that $(n,\frac{d}{2\la})$-expander graphs have a minimum degree $\frac{d}{2\la}$, which is also considerably weaker than that of $(n, d,\lambda)$-graphs (which are $d$-regular).

\subsection{Further implications}

Our main result also allows us to improve several other interesting results of Johanssen, Krivelevich and Samotij~\cite{JKS12}.

\subsubsection{Locally sparse expanders}

Johanssen, Krivelevich and Samotij~\cite{JKS12} studied constructions of tree-universal graphs which are locally sparse. Using probabilistic arguments (Lemma~6.2 in \cite{JKS12}) together with Theorem~\ref{thm:JKS}, they were able to prove the existence of a tree-universal graph with small clique number.
\begin{theorem}\cite[Theorem 2.4]{JKS12}\label{local}
There exists an absolute constant $c>0$ such that the following statement
holds. Let $n \in\N$ and let $r \in\N$ with $r \ge 5$. Then there exists a graph with clique number at
most $r$ that is $T(n,cn^{1/3-2/(r+2)}/ \log n)$-universal.
\end{theorem}

Moreover, they~\cite{JKS12} posed the question of finding constructions that are triangle-free or even have large girth whilst keeping tree-universality. A promising candidate is the celebrated triangle-free construction of an $(n, d, \la)$-graph due to Alon~\cite{A94} with $d=\Theta(n^{2/3})$ and $\la=\Theta(n^{1/3})$ or its generalizations to sparse expanders without short odd cycles~\cite{AK98,KS06}.
Theorem~\ref{th1} implies that these graphs are tree-universal.

\begin{theorem}
For $\De,r\in \N$ with $r\ge 1$ and sufficiently large $n \in\N$, there exists an $n$-vertex graph $G$ that is $T(n,\De)$-universal and contains no odd cycle of length at most $2r + 1$. Moreover, $G$ can be explicitly constructed.
\end{theorem}
Comparing with Theorem~\ref{local}, we also observe that using the same probabilistic arguments, Theorem~\ref{th2} (in place of Theorem~\ref{thm:JKS}) implies the existence of a graph with clique number at most $r$ that is $T(n,cn^{1/2-2/(r+2)}/ \log n)$-universal, where $r\ge 3$.
\subsubsection{Maker-Breaker Game}
An $(a : b)$ Maker-Breaker game is played on a finite hypergraph $H=(V, \mathcal{F})$ between two players,
Maker and Breaker. The game is played in turns, starting with Maker's turn. In each of
their turns, Maker claims $a$ and Breaker claims $b$ previously unclaimed vertices, respectively. Maker's objective is to claim all vertices of a hyperedge by the end of the game. In this case, Maker wins the game. Breaker's
objective is to claim at least one vertex in each hyperedge by the end of the game. In this
case, Breaker wins the game. The game ends when all vertices have been claimed, by which
time either Maker or Breaker have won. The numbers $a$ and $b$ are called the \emph{biases} of Maker and Breaker, respectively. We say that an $(a : b)$ Maker-Breaker game is Maker's win if Maker has a strategy that allows him to win the game regardless of Breaker's strategy, otherwise the game is Breaker's win. For a more detailed discussion, we refer the reader to \cite{BECK}.

Johanssen, Krivelevich and Samotij~\cite{JKS12} formulated a Maker-Breaker Expander Game.
\begin{definition}\cite[Maker-Breaker Expander Game]{JKS12}
For $n,\De \in\N$, the \emph{Maker-Breaker $(n,d)$-expander game} on a graph $G$ is the Maker-Breaker game on the hypergraph $H=(E(G), \mathcal{F})$, where $\mathcal{F}$ consists of all edge sets $F \< E(G)$ such that the subgraph $(V(G), F)$ is an $(n,d)$-expander.
\end{definition}
They showed that if the $(1 : b)$ Maker-Breaker expander game is played on an $(n, 15bd \log n)$-expander, then Maker can always secure all edges of an $(n, d)$-expander.
\begin{theorem}\cite{JKS12}\label{expgame}
There exists an absolute constant $n_0\in \N$ such that the following statement
holds. Let $n, b \in\N$ and $d\in \mathbb{R}$ satisfy $n \ge n_0$ and $d \ge 3$. Then the $(1 : b)$ Maker-Breaker
$(n, d)$-expander game is Maker's win on every $(n, 15bd \log n)$-expander.
\end{theorem}
Similarly they formulated a \emph{Maker-Breaker Tree-Universality Game} in which Maker tries to
claim a subset of the edges $F\<E(G)$ such that the subgraph $(V(G), F)$ is $\mathcal{T}(n,\De)$-universal. Using the tree-universality result in Theorem~\ref{thm:JKS}, they obtained the following corollary.
\begin{corollary}\cite{JKS12}
There exists an absolute constant $C>0$ such that the following statement
holds. Let $n,\De,b\in\N$ satisfy $\De \ge \log n$. Then the $(1 : b)$ Maker-Breaker $\mathcal{T}(n,\De)$-universality game is Maker's win on every $(n, d)$-expander with $d \ge Cb\De n^{
2/3}\log n$.
\end{corollary}
Combining Theorems~\ref{th2},~\ref{th1+} and~\ref{expgame}, we improve the bound on the expansion as follows.
\begin{corollary}
There exists an absolute constant $C>0$ such that the following statement
holds. Let $n,\De,b\in\N$ satisfy $\De \ge \log n$. Then the $(1 : b)$ Maker-Breaker $\mathcal{T}(n,\De)$-universality game is Maker's win on every $(n, d)$-expander with $d \ge \min\{\De n^{1/2}, \De^{5\sqrt{\log n}}\}Cb\log n$.
\end{corollary}

\subsection{Technical Contribution}

Our approach builds on previously known ones but also introduces new twists.
A crucial new idea is the \emph{reconstruction of tree array} (see Definition~\ref{def:tree} and Lemma~\ref{rec1}), which allows us to reconstruct local parts of the almost spanning tree so that we can embed a specific set of vertices of the tree to a (small) vertex set with good expansion property.
In particular, this is also the bottle-neck of the current proof, that is, the expansion property required in our proof comes from the need of reconstruction of paths of length $\sqrt{\log n}$, which can be understood as our requirement on the radius of the host graph.
Although the idea is natural, we are not aware of similar treatments in the literature.
Finally, we expect that the techniques developed here will find other applications in related problems, in particular, in embedding (connected) spanning substructures.

\subsection{Basic notation}
Given a graph $G$ and two vertex sets $A,B\subseteq V(G)$, we define $\Gamma_G(A):=\bigcup_{v\in A}N_G(v)$ and $N_G(A):=\Gamma(A)\setminus A$. We write $N_G(A,B)=N_G(A)\cap B$ and for a vertex $v$, let $d_G(v,B)=|N_G(v)\cap B|$. %In particular, let $e_G(A,B):=\sum_{v\in A}d_G(v,B)$.
Throughout the paper, we will often omit the subscript term $G$ to ease the notation, unless otherwise stated.

%\blue
{For $d,h\in \N$, a $d$\emph{-ary} tree \emph{of height} $h$ is a rooted tree in which every non-leaf vertex has $d$ children and every leaf is of distance exactly $h$ from the root.}
Given a tree $T$, let $T'$ be a subtree obtained by removing all leaves in $T$.
%Then we call the leaves in $T'$ \emph{second-level} leaves.
A \emph{pendant star} in $T$ is a maximal star centered at a leaf vertex of $T'$, and the unique neighbour of the center in $T'$ is the \emph{root} of the pendant star (see Figure~\ref{fig1}).
A path $P$ in $T$ is a \emph{bare} path if all internal vertices of $P$ have degree exactly two in $T$.
A \emph{caterpillar} in $T$ consists of a bare path in $T'$ as the \emph{central} path
and a (possibly empty) set of leaves in $T$ attached to internal vertices of the central path, where \emph{branching} vertices are the internal vertices attached with at least one leaf and we call each such leaf a
\emph{leg}. Moreover, the \emph{length} and \emph{ends} of a caterpillar refer to the length and ends of the corresponding central path, respectively. Given a path $P$ and two distinct vertices $a,b\in V(P)$, denote by $P(a,b)$ the subpath that connects $a$ and $b$ in $P$.

Given two graphs $H,G$, an embedding of $H$ in $G$ is an injective mapping $\phi: V(H)\rightarrow V(G)$ such that $\{\phi(u),\phi(v)\}\in E(G)$ for every edge $uv\in E(H)$. We often use $\phi|_H$ to denote an embedding of $H$ (in $G$).
%$\phi|_{T-C}$:

\section{Outline of the proofs}

In this section we give an outline of our proofs, and compare with previous approaches.
%\subsection{Division into cases}
%The widely-used embedding strategy of a tree $T$ varies depending on the structure of $T$.
Firstly, what is common for previous proofs~\cite{JKS12, Kri10, Montgomery19} is to distinguish the trees according to the number of leaves.
The following key observation is due to Krivelevich~\cite{Kri10}.
\begin{lemma}\cite{Kri10}
\label{lem:fact}
For any integers $n, k > 2$, an $n$-vertex tree either has at least $\frac{n}{4k}$
leaves or a collection of at least $\frac{n}{4k}$ vertex disjoint bare paths each of length $k$.
\end{lemma}

%Based on this, a framework for embedding a tree $T$ goes as follows: find within it some leaves or bare paths, remove these from $T$ to get a forest $T'$ , embed $T'$ using almost-spanning tree embedding results (see Theorem~\ref{almost2} or Theorem~\ref{almost}), and then embed the removed leaves or bare paths from.
%More often, the major challenge is to embed the deleted leaves and bare paths, which must necessarily cover exactly the remaining vertices in the graph. However,...\todo{rewrite}
Based on this fact, previous approaches distinguish the trees in $\mathcal{T}(n, \Delta)$ into two classes.
If a tree $T$ has many leaves, then we remove these leaves and obtain a subtree $T'$.
The subtree $T'$ can be embedded in the host graph by almost-spanning tree embedding results (e.g. Theorem~\ref{almost}), and then the leaves of $T$ can be further embedded to the remaining vertices by matching-based arguments.

We first outline the proof of Theorem~\ref{thm:JKS} in~\cite{JKS12}.
The proof of Theorem~\ref{thm:JKS} uses another version of Lemma~\ref{lem:fact}, that is, $T$ either has many leaves, or has \emph{one} long bare path.
The authors of~\cite{JKS12} actually went one step further, namely, they further distinguished the \emph{many-leave} case into two cases depending on $T'$ -- either $T'$ has a long bare path, or $T'$ has many leaves.
Then the three cases are solved by different methods.
In particular, they used a result on Hamiltonian-connectedness to embed long bare path (after embedding $T'$).

Montgomery's proof of Theorem~\ref{thm:RM19} in~\cite{Montgomery19} is far more sophisticated, when dealing with the \emph{many-bare-path} case.
In fact, for random graphs, in~\cite{Alon07} Alon Krivelevich and Sudakov already observed that one can obtain spanning tree embeddings for trees with many leaves, by first applying their result on almost spanning trees and then using a Hall-type matching argument to embed the leaves of $T$.
This is easily done by the \emph{multi-round exposure} technique, as one can reveal new random edges.
However, the case when $T$ has many (long) bare paths is significantly more challenging.
Montgomery developed a novel method, \emph{absorption using bipartite template}, to complete the embedding of a disjoint union of (bare) paths, which we shall use as well in our proof.
This new method has found many applications in embedding problems in sparse graphs and hypergraphs.

Now we discuss briefly on our proof ideas.
In fact, for the many-bare-path case we could follow the embedding strategy (and some of the results) of Montgomery.
However, since we do not work with random graphs, the many-leave case is no longer free (modulo the almost spanning tree embedding results), and in fact becomes challenging.
Indeed, using a result of~\cite{JKS12}, one can partition the graph $G$ into multiple blocks while reserving the expansion property in each block (i.e., $G$ expands into each block).
Thus, letting $C$ be the leaves of $T$ and $B$ be the set of the parents of $C$, one can embed $T-C$ by an embedding $\phi$ in a (huge) block of $G$ so that what is left for the image of $C$ is a block $V'$ with good expansion property (namely, $G$ expands into $V'$).
However, this ``one-sided'' expansion is too weak for us to establish a star-matching result on the bipartite graph $G[\phi(B), V']$.
The hope is to strengthen the embedding $\phi$ so that $\phi(B)$ also enjoys a good expansion property and thus the matching-type result can be applied to $G[\phi(B), V']$.
To achieve this, it is natural to check $T'$ similarly as in~\cite{JKS12} and split into further cases.

%In our proof of Theorem~\ref{th2}, we will instead use a refined classification as follows, which has already appeared in our recent work~\cite{HHPWWY}.
\begin{corollary}[first-round deletion]\label{newdiv}
For any integers $n, \Delta, k > 2$, an $n$-vertex tree $T\in \mathcal{T}(n,\Delta)$ either has at least $\frac{n}{4k\De}$
pendant stars or a collection of at least $\frac{n}{4k\De}$ vertex disjoint caterpillars each of length $k$.
\end{corollary}

\begin{figure}[htb]
\center{\includegraphics[width=10.4cm] {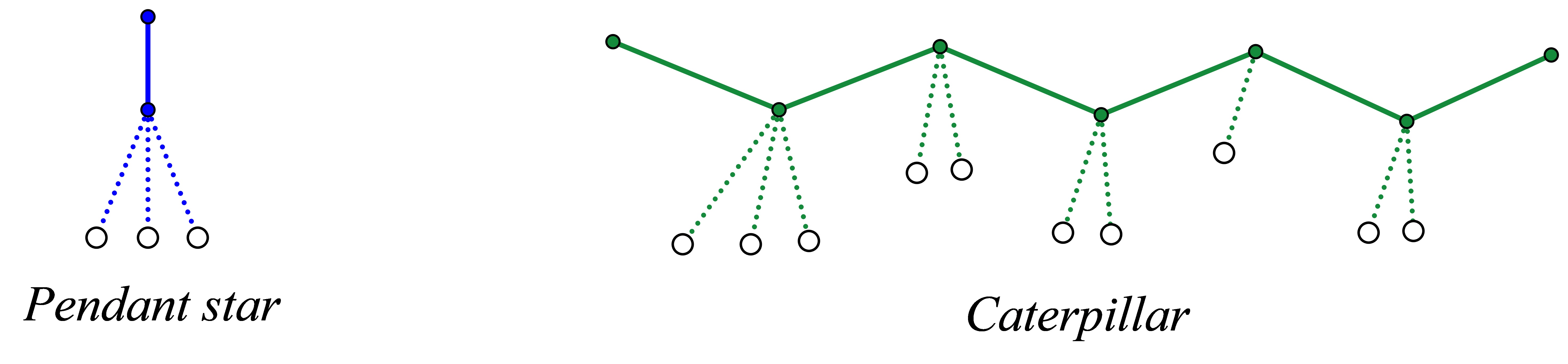}}
\caption{}
\label{fig1}
\end{figure}

%This appears to be enough for the proof of Theorem~\ref{th2}.
This actually gives threes cases (recall that $T'$ denotes the subtree obtained by removing all leaves of $T$):
\begin{enumerate}
\item the many-bare-path case ($T$ has many bare paths),
\item the many-pendant-star case (both $T$ and $T'$ have many leaves), and
\item the many-caterpillar case ($T$ has many leaves but $T'$ does not).
\end{enumerate}
As mentioned above, (1) can be treated by the method of Montgomery, while (2) and (3) need new ideas.
The novel part of our approach lies in the construction of the almost spanning tree, so that roughly speaking (a large set of) the leaves of $T'$ can be embedded to a block enjoying expansion property (and then a matching-type argument can embed the leaves of $T$ to the remaining vertices of $G$).
Let $B'$ be the set of the leaves of $T'$ and $A$ be the set of the parents of $B'$.
Then choose $B\subseteq B'$ such that $|B|=|A|$ and vertices of $B$ have distinct parents, namely, $T[A\cup B]$ is a matching. Let $C$ be the set of leaves incident to $B$.
We start with a random partition of $V(G)$ into three blocks $V_1,V_2, V_3$, of size
\[
|V_1|=n-|B|-|C| + 22\Delta m, \quad |V_2|=|B| \quad\text{ and }\quad |V_3|=|C|-22\Delta m,
\]
where $m=n/2d$ (recall that every $(n,d)$-expander is $m$-joined).
%$|B'|$, $|C|$ and $n-|B'|-|C|$.
We first find an embedding $\phi$ that maps $T'=T-(B\cup C)$ to $V_1$ by Theorem~\ref{almost}, with a leftover of $22\Delta m$ vertices.
Next we try to match $\phi(A)$ and $V_2$, which, if successful, would embed $B$ to $V_2$ which has good expansion property.
However, this is asking for a \emph{perfect matching} between $\phi(A)$ and $V_2$, and as we only have one-sided expansion property, this may not be possible (actually, this is by the same reason why we cannot finish the final star-matching mentioned before).
Now, a quick solution is to enlarge $V_2$ by part of $V_3$, so that the desired matching only covers $\phi(A)$ and is not perfect.
More precisely, let $V_2'$ be obtained from adding $m$ vertices of $V_3$ to $V_2$.
Since $G$ is $m$-joined, it is easy to find a matching in $G[\phi(A), V_2']$ that covers $\phi(A)$.
Now what we want is that the expansion property from $V_3$ to $V_2$ is not too much damaged, which can be guaranteed if $m = o(\frac dn|B|)$.
This strategy leads to a condition $d = \Theta(\De\sqrt{n})$.
Indeed, this (with some technical work) is enough for establishing Theorem~\ref{th2}.

To prove Theorem~\ref{th1+}, we need a more complex approach.
Note that we can find a maximum matching $M$ in $\phi(A)$ and $V_2$.
It may not be perfect, but by the $m$-joinedness, there are less than $m$ unmatched vertices on both sides.
Then a na\"ive idea is to replace the vertices in $\phi(A)$ as follows.
For $v\in \phi(A)\setminus V(M)$, we search for a path of length two from its parent to $V_2\setminus V(M)$, whose internal vertex is from $V_3$.
This path would be almost impossible to find -- because of its short length.
What we actually do is to try to find a much longer path (indeed, of length $h:=\sqrt{\log n}$).
To make sure the distance of $v$ and its ancestor is large enough, at the beginning of the proof, we apply Lemma~\ref{lem:fact} to $T_h$, a tree obtained from $T$ by recursively removing all leaves $h$ times.
First assume that $T_h$ has many leaves.
We redefine $B$ and $C$ such that vertices of $B$ have distinct ancestors with distance $h$, and denote the set of their ancestors by $A_h$, and $C$ is defined as the leaves of $T$ in $N(B)$.
Then the first step is to embed $T-(B\cup C)$ in a block of $G$.
When embedding $B$, we first find a maximum matching $M$ between $\phi(A)$ and $V_2$.
Set $A_h'\subseteq A_h$ be the set of ancestors of vertices in $\phi(A)\setminus V(M)$ and we try to connect vertices of $A_h'$ and $V_2\setminus V(M)$, by vertex-disjoint paths of length $h-1$, using part of vertices of $V_3$ (reserved for this connection).
After the connection is done, we also need to rebuild the pendant trees rooted at these paths (they were embedded by $\phi$ but just discarded, see Figure~\ref{fig2}).
After the reconstruction of the pendant trees, $B$ is completely embedded to $V_2$.
Now a matching-type argument finishes the embedding of $T$.
Second assume that $T_h$ has many bare paths.
If e.g. half of these paths are also bare paths in $T$, then we can follow Montgomery's approach.
Otherwise, half of these paths have pendant trees rooted at their internal vertices.
There are two further subcases: i) a quarter of the paths have the pendant trees have height at least two; ii) a quarter of the paths are caterpillars.
For i), we shall first embed the subtree obtained from $T$ excluding pendant stars (one star for each path).
For ii), we first embed the subtree of $T$ with those caterpillars removed.
In both cases, to finish the embedding, a similar (and slightly more complicated) reconstruction scheme works so that we can embed the centers of the stars or the roots of the ``first'' legs of all caterpillars to a block with expansion property, leaving the final star-matching possible to achieve.

\section{Preliminaries}
In this section we collect various tools which have been used in previous works~\cite{Haxell01,JKS12,Kri10,Montgomery19} on embedding (almost) spanning trees, and also present some new results to aid our embedding.

\subsection{Random partition of expanders}
The following result of Johanssen, Krivelevich and Samotij\\~\cite{JKS12} allows us to partition our expander into small expanders. For our convenience, we give a slightly stronger version of their result, whose proof identically follows from that of Lemma~3.4 in \cite{JKS12}.
\begin{lemma}\cite[Lemma 3.4]{JKS12}\label{randompart}
There exists an absolute constant $n_0\in \N$ such that the following statement holds. Let $k,n\in \N$ and $d\in \mathbb{R}^+$ satisfy $n\geq n_0$ and $k\leq \log n$. Furthermore, let $n,n_1,\ldots,n_k\in \N$ satisfy $n=n_1+\ldots+n_k$ and let $d_i:=\frac{n_i}{5n}d$ satisfy $d_i\geq 2\log n$ for all $i\in \{1,\ldots,k\}$. Then, for any graph $G$ which $d$-expands into a vertex set $W$, with $|W|=n$, the set $W$ can be partitioned into $k$ parts $W_1,\ldots, W_k$ of sizes $n_1,\ldots,n_k$ respectively, such that $G$ $d_i$-expands into $W_i$ for every $i\in[k]$.
%\begin{enumerate}
%  \item [$(1)$]\textcolor[rgb]{0.00,0.00,1.00}{ $d_G(v,W_i) \ge \frac{n_i}{2n}d_G(v, W)$ holds for every $v \in V(G)$ and $i\in[k]$;}
%  \item [$(2)$] F, .
%\end{enumerate}
\end{lemma}

\subsection{Building star-matchings}
We will use two results on matchings and here we first give a notion of generalized matching as follows.
\begin{definition}
Given a bipartite graph $G=(A,B,E)$ and an function $f: A\rightarrow \N$ with $\sum_{u\in A}f(u) = |B|$, an $f$-\emph{matching  from} $A$ \emph{into} $B$ is a collection of vertex-disjoint stars $\{S_u:u\in A\}$ in $G$ such that $S_u$ has $u$ as the center and exactly $f(u)$ leaves inside $B$.
\end{definition}
More often we call such an $f$-matching a star-matching. The following two results give sufficient conditions on the existence of star-matchings.
\begin{lemma} \cite[Lemma~3.10]{JKS12}\label{star}
Let $d, m\in \mathbb{N}$ and $G$ be a graph. Suppose that two disjoint
sets $A, B \subseteq V(G)$ satisfy the following conditions:
\begin{enumerate}
  \item [$(1)$] $|N_G(X)\cap B| \ge d|X|$ for all $X \subseteq A$ with $1\le |X| \le m$;
  \item [$(2)$] $e_G(X, Y) > 0$ for all $X\subseteq A, Y\subseteq B$ with $|X| = |Y |=m$.
  \item [$(3)$] $|N_G(w) \cap A| \ge m$ for all $w \in B$.
\end{enumerate}
Then, for every $f: A\rightarrow \{1, \ldots, d\}$ with $\sum_{u\in A}f(u) = |B|$, there exists an $f$-matching from $A$ into $B$.
\end{lemma}

\begin{lemma}\label{star1}
Let $d,m\in \mathbb{N}$ and $G$ be a graph. Suppose that two disjoint
sets $A, B \subseteq V(G)$ satisfy that
\begin{enumerate}
  \item [$(1)$] $|N_G(X)\cap B| \ge d|X|, N_G(Y)\cap A| \ge d|Y|$ for all $X \subseteq A, Y\<B$ with $1\le |X|,|Y| \le m$;
  \item [$(2)$]\label{equ1} $e_G(X, Y) > 0$ for all $X\subseteq A, Y\subseteq B$ with $|X| = |Y |=m$
\end{enumerate}
Then, for every $f: A\rightarrow \{1, \ldots, d\}$ with $\sum_{u\in A}f(u) = |B|$, there exists an $f$-matching from $A$ into $B$.
\end{lemma}
\begin{proof}
To build a desired star-matching, we shall show a generalized Hall's condition (see, e.g., Bollob\'{a}s \cite{Bo98})), that is,
\begin{eqnarray}\label{eq4}
|N(X)\cap B|\ge \sum_{v\in X}f(v)~\text{for every}~X\<A.
\end{eqnarray}
From the assumption that every set $X$ of at most $m$ vertices in $A$ has at least $d|X|$ neighbors in $B$, we know that \eqref{eq4} holds when $|X|\le m$. It remains to consider the following two cases depending on the size of $|X|$:
\begin{itemize}
  \item $m<|X|\le |A|-m$: By the assumption~$(2)$ and $|B|=\sum_{v\in A}f(v)$, we have \[|N(X)\cap B|> |B|-m\ge \sum_{v\in A}f(v)-\sum_{v\in A\setminus X}f(v)=\sum_{v\in X}f(v),\] where the second inequality follows as $m\le |A|-|X|\le\sum_{v\in A\setminus X}f(v)$.
  \item $|A|-m<|X|\le |A|$: Let $T:=B\setminus N(X)$. We may assume that $T\neq\emptyset$ for otherwise we are done. As there is no edge between $X$ and $T$, we obtain that $|T|<m$ and $N(T)\cap A\<A\setminus X$. Then it follows by assumption that $d |T|\le|N(T)\cap A|\le|A\setminus X|\le\sum_{v\in A\setminus X}f(v)$. Thus we have
      \[|N(X)\cap B|=|B|-|T|=\sum_{v\in X}f(v)+\sum_{v\in A\setminus X}f(v)-|T|>\sum_{v\in X}f(v).\]
\end{itemize}
In all cases, condition~\eqref{eq4} holds and this yields a desired star-matching covering $A\cup B$. %which together with $\phi|_{T-C}$ forms an embedding of $T$.
\end{proof}

\subsection{Connecting pairs of vertices in expanding graphs}
The following two results are used to build a family of vertex-disjoint paths each connecting a prescribed pair of vertices whilst covering all the fixed vertices.

\begin{theorem}\cite{Montgomery14}
\label{pathcover}
Let $n$ be sufficiently large and let $\ell\in \N$ satisfy $\ell\geq 10^3\log^2n$ and $\ell|n$. Let a graph $G$ contain $n/\ell$ disjoint vertex pairs $(x_i,y_i)$ and let $W=V(G)\setminus (\cup_i\{x_i,y_i\})$. Suppose $G$ $d$-expands into $W$, where $d=10^{10}\log^4 n/\log\log n$. Then we can cover $G$ with $n/\ell$ disjoint paths $P_i$, each of length $\ell-1$, so that, for each $i$,  $P_i$ is an $(x_i,y_i)$-path.
\end{theorem}

\begin{theorem}[Path-cover]\label{pathcover2}
There exists $C>0$ such that the following holds for any integer $\ell\ge 200$ and sufficiently integer $n$ with $\ell|n$. Suppose an $n$-vertex graph $G$ contains $n/\ell$ disjoint vertex pairs $(x_i,y_i),i\in[n/\ell]$ and let $W=V(G)\setminus (\cup_i\{x_i,y_i\})$. If $G$ $d$-expands into $W$ %\todo{replace by min degree+$m$-joined}
 with $d=C\ell\sqrt{n}$, then we can cover $G$ with $n/\ell$ disjoint paths $P_i$, each of length $\ell-1$, so that, for each $i$,  $P_i$ is an $(x_i,y_i)$-path.
\end{theorem}

We also need a technical lemma which will be used to (re)construct a forest in the proof of Theorem~\ref{th1+}. We first introduce a notion of tree array.
%\subsection{Reconstruction of subtrees}

\begin{definition}[Tree array]
\label{def:tree}
Given a graph $G$, a subset of vertices $W$, a collection $\mathcal{I}$ of disjoint pairs of vertices from $V(G)\setminus W$ and integers $s,\Delta$, we call that $G$ contains a $(W,\mathcal{I},s,\Delta)$\emph{-tree array} if there exist
\begin{itemize}
	\item a family $\mathcal{P}=\{P_{xy}: (x,y)\in \mathcal{I}\}$ of mutually vertex-disjoint paths each of length $s$, where the internal vertices of all paths $P_{xy}$, denoted as $\mathsf{Int}(\mathcal{P})$, lie inside $W$;
	\item a family $\{T_v\}_{v\in \mathsf{Int}(\mathcal{P})}$ of mutually vertex-disjoint rooted trees inside $W$, such that for every path $P_{xy}$ and every internal vertex $v$, $T_v$ is a $\Delta$-ary tree rooted at $v$ with height $s$, which is vertex disjoint from $V(\mathcal{P})\setminus\{v\}$.  %such that each adjuster in $\mathcal{D}$ has at least one end which is connected to $R$ by a subpath from a path in $\mathcal{P}$, and all the paths are disjoint from all center sets of the adjusters in $\mathcal{D}\cup \{A\}$. It is easy to observe that $\abs{\mathcal{P}}\le \abs{\mathcal{D}}$.
    \end{itemize}
\end{definition}

%\begin{lemma}[Tree array]\label{rec0}
% For all $\Delta\in \mathbb{N}$, let $G$ be an $(n,d,\lambda)$-graph with \[\frac{\lambda}{d}\le\frac{1}{\De^{5\sqrt{\log n}}}\] and $W$ be a set of at least $\sqrt{\frac{\lambda}{d}}n$ vertices in $G$. If for every vertex $v\in V(G)$, $d(v,W)\ge \frac{|W|}{2n}d$, then for every $s\in [\sqrt{\log n}-1,2\sqrt{\log n}]$ and every family $\mathcal{I}$ of at most \textcolor[rgb]{1.00,0.00,0.00}{$\frac{|W|}{10000s\Delta^{s+1}}$} disjoint pairs of vertices from $V(G)- W$, $G$ contains a $(W,\mathcal{I},s,\Delta)$-tree array.
%\end{lemma}

\begin{lemma}\label{rec1}
For all $s, t, m,d_1,\Delta\in \mathbb{N}$ with $d_1\ge \De+2$ and $s\ge 2\lceil\frac{\log 2m}{\log (d_1-1)}\rceil+1$, let $G$ be an $m$-joined graph and $W\<V(G)$. If $G$ $d_1$-expands into $W$ and $|W|>10d_1m+ t(s+1)\De^{s+1}$, then for every family $\mathcal{I}$ of at most $t$ disjoint pairs of vertices from $V(G)\setminus W$, $G$ contains a $(W,\mathcal{I},s,\Delta)$-tree array.
\end{lemma}

\begin{corollary}\label{rec}
 For sufficiently large $n\in \N$ and $d,\De\in \N$ with $n>d\ge\De^{5\sqrt{\log n}}$, let $G$ be an $m$-joined graph with \[m\le\frac{n}{2d}\] and $W$ be a set of at least $\frac{n}{\sqrt{d}}$ vertices in $G$. If $G$ $\De^{2\sqrt{\log n}}$-expands into $W$, then for every $s\in [\sqrt{\log n}-1,2\sqrt{\log n}-1]$ and every family $\mathcal{I}$ of at most $m$ disjoint pairs of vertices from $V(G)- W$, $G$ contains a $(W,\mathcal{I},s,\Delta)$-tree array.
\end{corollary}
\begin{proof}
Let $h=\sqrt{\log n}$ and $d_1=\De^{2h}$. Then $G$ $d_1$-expands into $W$ and $|W|\ge \frac{n}{\De^{5h/2}}$.
Note that $s\in [h-1,2h-1]$ and \[2\left\lceil\frac{\log 2m}{\log (d_1-1)}\right\rceil+1 \le2\left\lceil\frac{h^2-\log\De^{ 5h}}{\log (\De^{2h}-1)}\right\rceil+1 \le h-1 .\] Since $n$ and also $h$ are sufficiently large, we have that \[10d_1m+ m(s+1)\De^{s+1}\le\frac{n}{\De^{5h}}(10\De^{2h}+10+2h\De^{2h})\le\frac{n}{\De^{3h}}(2h+11)\le\frac{n}{\De^{5h/2}}\le|W|.\] We can apply Lemma~\ref{rec1} with $t=m$ to obtain a desired $(W,\mathcal{I},s,\Delta)$-tree array.
\end{proof}

\section{Proof of Theorem~\ref{th1+}}
In this section, we prove Theorem~\ref{th1+}. Throughout the proof, we write
\begin{eqnarray}\label{eq1}
h=\lceil\sqrt{\log n}\rceil~\text{and}~ k=\lceil\log^3 n\rceil.
\end{eqnarray}
Moreover, we may take \begin{eqnarray}\label{eq11}
d=\De^{5\sqrt{\log n}},~ m:=\frac{n}{2d}\le\frac{n}{2\De^{5h-5}}
\end{eqnarray}
 and fix $G$ to be an $(n,d)$-expander. Then $G$ is $m$-joined.

Given a tree $T\in \mathcal{T}(n,\Delta)$, we write $T_0:=T$ and $T_i$ for the subtree of $T_{i-1}$ obtained by removing all leaves in $T_{i-1}$, where $i\in [h]$. Let $L_{i}$ be the set of all leaves in $T_i$ and $n_i:=|V(T_i)|, i\in\{0,1,\ldots,h\}$. Then it is easy to show that $n_i(1+\Delta+\ldots+\Delta^i)\ge n$ for every $i\in[h]$. In particular,
\begin{eqnarray}\label{eq2}
n_h\ge \frac{n}{\De^{h+1}}.
\end{eqnarray}
 The proof of Theorem~\ref{th1+} is split into two cases depending on the structure of subtree $T_{h}$.

%Note that by~\eqref{}, when $d$ is sufficiently large, then $n/d$ and also $m$ are sufficiently large, and
%\begin{eqnarray}\label{eq-n-large}
%n/d\ge m^{200} \quad\mbox{ and }\quad d\ge m^{50s}.
%\end{eqnarray}
%Also, for sufficiently large $d$, since $\rho(x)$ is decreasing in the interval $[\eps_2d/2,n]$, we have that for every $\eps_2d/2\le x\le n$,
%\begin{eqnarray}\label{eq-eps-not-small}
%     \rho(x)\ge \rho(n)\ge\frac{1}{m}.
%\end{eqnarray}

\subsection{$T_{h}$ has many leaves}

We first consider the case when $T_{h}$ has a set $L_{h}$ of at least $n_h/4k$ leaves. Note that by definition, for every $u\in L_h$, there exists a choice $f(u)\in L_1$ which is of distance exactly $h-1$ from $u$ (we choose an arbitrary one if there are more than one choice). Let $B\<\{f(u): u\in L_h\}$ be a set of exactly $n_{h}/4k$ vertices and $C\subseteq L_0$ be the set of descendants of vertices in $B$. Denote by $A_2$ the neighborhood of $B$ inside $L_2$ and for every $i=3,4,\ldots,h$, let $A_i$ be the neighborhood of $A_{i-1}$ inside $L_i$. Then it is easy to see that $|A_i|=|B|=n_{h}/4k$ for every $i=2,3,\ldots,h$.

Let $T':=T-B\cup C$. Now our proof proceeds in the following steps. %For given and , we choose constants $t$ such that

\begin{figure}[htb]
\center{\includegraphics[width=10.4cm] {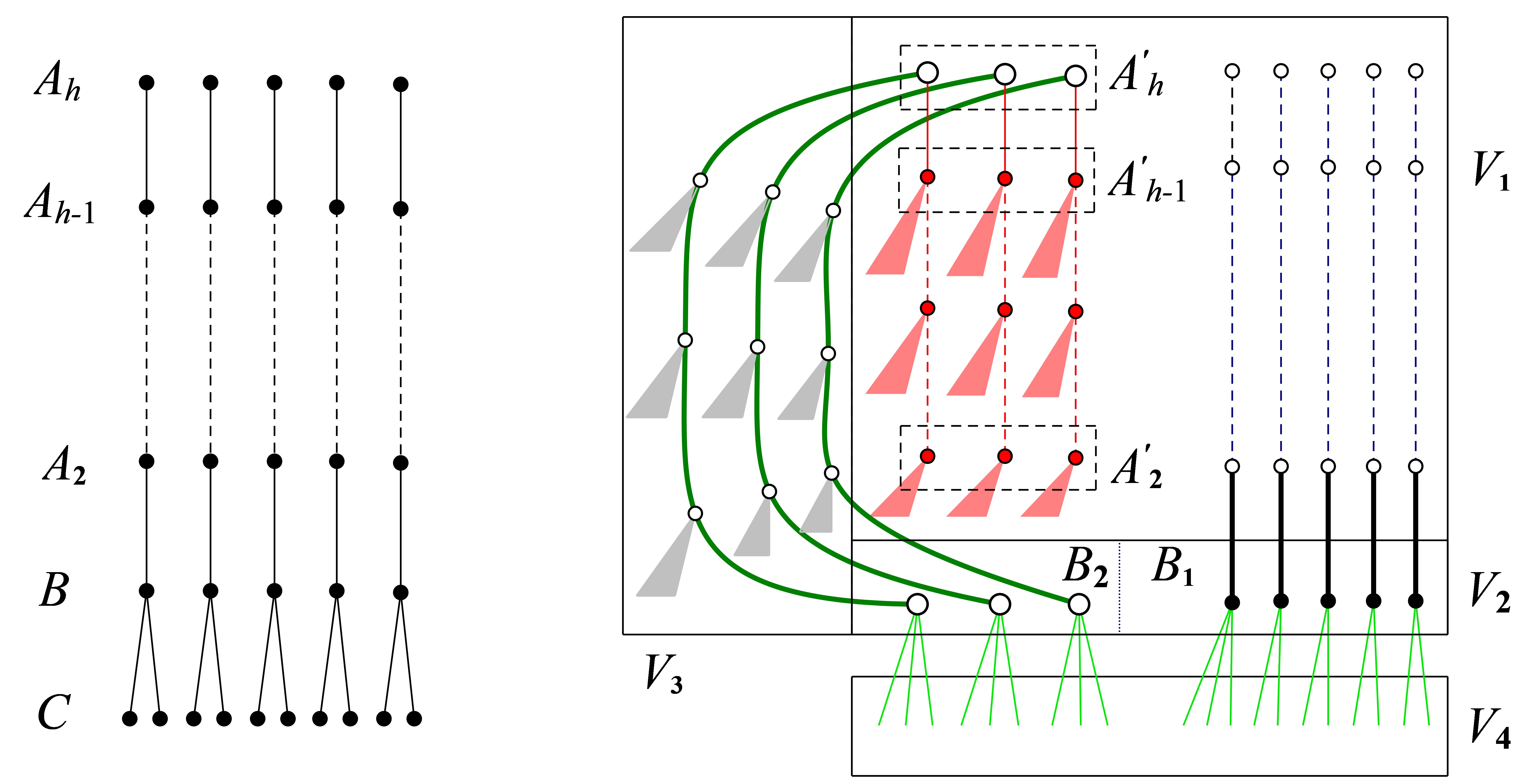}}
\caption{The red trees have height at most $h$ and would be replaced by the green paths attached with gray subtrees inside $V_3$. }
\label{fig2}
\end{figure}

\medskip

\noindent \textbf{Phase $0$. Partition $V(G)$}. We first randomly partition $V(G)$ into four parts $V_1,V_2,V_3,V_4$ such that
\begin{eqnarray}\label{eq3}
|V_1|=n-|B|-|C|+22\De m,~|V_2|=|B|~\text{and} ~|V_3|=|V_4|=\frac{|C|}{2}-11\De m.
\end{eqnarray}

We can easily check that
%\todo{This is where $k$ is used, and upper bound $|V_3|<\frac{n}{k\De^{h+1}}$}
\begin{eqnarray}
|V_2|=|B|=\frac{n_{h}}{4k}\overset{\eqref{eq2}}{\ge} \frac{n}{4k\De^{h+1}} \overset{\eqref{eq11}}{\ge}\frac{n}{\sqrt{d}}, \label{eq31} \\
|V_3|=|V_4|=\frac{|C|}{2}-11\De m\ge\frac{|B|}{2}-11\De m\ge\frac{n}{\De^{2h}}-11\De\frac{n}{2\De^{5h-5}}\overset{\eqref{eq11}}{\ge}\frac{n}{\sqrt{d}}, \label{eq32} \\
|V_1|=n-|B|-|C|+22\De m >\sum_{i=2}^h|A_i|=(h-1)|B|\overset{\eqref{eq31}}{>}\frac{n}{\sqrt{d}}.\label{eq33}
\end{eqnarray}
Since $G$ is an $(n,d)$-expander  and
\begin{eqnarray}\label{eq34}
d_i:=\frac{|V_i|}{5n}d\ge \frac{\sqrt{d}}{5}>\max\{\De^{2\sqrt{\log n}},2\log n\}~\text{for every}~ i\in [4],
\end{eqnarray}
 by Lemma~\ref{randompart} applied to $G$ with $W=V(G)$, there exists such a partition with the following property:
\stepcounter{propcounter}
\begin{enumerate}[label = ({\bfseries \Alph{propcounter}\arabic{enumi}})]
%       \item\label{p1} For each $i\in [4]$ and $v\in V(G)$, $d(v,V_i)\ge \frac{|V_i|}{2n}d$.

       \item\label{p2} $G$ $d_i$-expands into $V_i$ for all $i\in[4]$.
\end{enumerate}

\medskip
\noindent \textbf{Phase $1$. Embed $T'$ in $G[V_1]$}.
By property~\ref{p2}, we can apply Theorem~\ref{almost} to $G[V_1]$ to find an embedding $\phi|_{T'}$ that maps $T'$ to $V_1$. Indeed, as $d_1\ge 2 \De$ in \eqref{eq34}, it suffices to ensure that the order of $T'$ is no larger than $|V_1| - 4\Delta \lceil\frac{|V_1|}{2d_1}\rceil$. This easily follows from \eqref{eq33} that \[|V_1| -|T'|=22\De m\overset{\eqref{eq11}}{>} 4\De\left\lceil\frac{5n}{2d}\right\rceil=4\De\left\lceil\frac{|V_1|}{2d_1}\right\rceil.\]

\medskip
\noindent \textbf{Phase $2$. Embed $B$ into $V_2$}. We first build a maximal matching $M$ between $\phi(A_2)$ and $V_2$. Denote by $A_2', B_2$ the sets of vertices in $\phi(A_2)$ and $V_2$, respectively, that are not covered by $M$. Since $G$ is $m$-joined, we have \[|A_2'|=|B_2|< m.\]
In order to finish the embedding of $B_2$, we shall first prune subtrees and then regrow them by using tree arrays.
For each $i\in[3,h]$, let $A_i'$ be the set of vertices in $\phi(A_i)$ which are ancestors of vertices in $A_2'$ (see Figure~\ref{fig2}).
Moreover, for every $v\in A_2'$ and its ancestor $u\in A_h'$, denote by $P_{u,v}$ the unique $(u,v)$-path in $\phi|_{T'}$. Then all these $P_{u,v}$ are pairwise disjoint and the removal of $\cup _{v\in A_2'}E(P_{u,v})$ from $\phi|_{T'}$ would yield a family of disjoint trees $T_x$ with root $x\in \cup_{i=2}^{h-1} A'_i$, each of height at most $h-1$.

Let $\mathcal{I}$ be a family of $|A_2'|$ disjoint ordered pairs $(u,w)\in A_h'\times B_2$ of vertices. % where $u$ is an ancestor of $v$.
Note that we have
%\todo{upper bound $|V_3|$$<\frac{n}{k\De^{h+1}}$}
\[
d_3\ge\De^{2\sqrt{\log n}},~|V_3|\ge \frac{n}{\sqrt{d}} ~ \text{and}~ |\mathcal{I}|< m
\]
by~\eqref{eq32} and~\eqref{eq34}.
Then by property~\ref{p2} and Corollary~\ref{rec} with $s=h-1$ and $W=V_3$, we obtain a $(V_3,\mathcal{I},h-1,\Delta)$-tree array. Then one can obtain an embedding $\phi|_{T-C}$ from $\phi|_{T'}$ by
\stepcounter{propcounter}
\begin{enumerate}[label = ({\bfseries \Alph{propcounter}\arabic{enumi}})]
       \item\label{prun} first deleting all vertices in $ A_{h-1}'$ and their descendants (in $\phi|_{T'}$, the trees in red in Figure~\ref{fig2});

       \item\label{rebuild} then connecting all pairs $(u,w)\in \mathcal{I}$ via vertex-disjoint paths $Q_{u,w}$ each of length $h-1$; and at every vertex $z\in V(Q_{u,w})\setminus \{u\}$, we regrow a desired copy of $T_x$ for some $x\in V(P_{u,v})\setminus \{u\}$.
\end{enumerate}
Observe that such paths $Q_{u,w}$ and trees in \ref{rebuild} can be easily obtained from the $(V_3,\mathcal{I},h-1,\De)$-tree array (by taking subtrees from $\De$-ary trees if necessary).

Hence it remains to embed $C$ into the remaining set $L$ of vertices with $L=V(G)-\phi(T-C)$.

\medskip
\noindent \textbf{Phase $3$. Build a star-matching between $V_2$ and $L$}.

To finish the embedding of $T$, we define an auxiliary function $f:V_2\rightarrow [\De]$ by taking $f(v)=d_T(\phi^{-1}(v))-1$, that is, the number of leaves we need to attach to every $v\in V_2$. Then $|L|=|C|=\sum_{v\in V_2}f(v)$. Note that $V_4\<L$ and by property~\ref{p2}, we observe that for every $X\< L$ with $|X|\le m<\frac{5n}{2d}=\frac{|V_2|}{2d_2}$,
\[|N_G(X,V_2)|\ge d_2|X|\overset{\eqref{eq34}}{\ge} \De|X|\] and similarly
$|N_G(Y,L)|\ge |N_G(Y,V_4)|\ge d_4|Y|\ge \De|Y|$ for every $Y\< V_2$ with $|Y|\le m$. By Lemma~\ref{star1} with $A=V_2,B=L$, we obtain a desired $f$-matching, which together with $\phi|_{T-C}$ completes an embedding of $T$.

\subsection{$T_{h}$ has many long bare paths}

In this section, we consider the case when $T_h$ has $\lfloor\frac{n_h}{4k}\rfloor$ vertex-disjoint bare paths, say $P_i, i\in\lfloor\frac{n_h}{4k}\rfloor$ each of length $k=\lceil\log^3n\rceil$. Let $x_i,y_i$ be the ends of $P_i$, for every $i\in\lfloor\frac{n_h}{4k}\rfloor$. Let $Q_i\<P_i$ be a subpath of length $k-4h$, with each of ends of distance exactly $2h$ from $x_i$ and $y_i$, respectively. Given an internal vertex $v\in V(Q_i)$, we say $Q_i$ \emph{sees} a pendant star $S$ \emph{from} $v$ in $T$ if there is a path in $T\setminus Q_i$ connecting $v$ to the root of $S$. Let $r=\lfloor\frac{n_h}{8k}\rfloor$.
In the forthcoming subsections, we shall consider three subcases depending on the surroundings of $Q_i$:
\begin{enumerate}[label= \textbf{Case \Alph*}]
  \item  at least $r$ paths $Q_i$ are bare in $T$;
  \item  at least $r/2$ paths $Q_i$ form caterpillars in $T$ with at least one leg;
  \item  at least $r/2$ paths $Q_i$ see pendant stars in $T$.
\end{enumerate}

 \subsubsection{Many long bare paths $Q_i$ in $T$: \textbf{Case A}}

We may assume that $Q_1,Q_2,\ldots,Q_{r}$ are bare paths in $T$, where each $Q_i$ has ends $a_i,b_i$. Let $T'=T-\cup_{i\in[r]}V(Q_i-\{a_i,b_i\})$. Recall that every $Q_i$ has length $k':=k-4h$. Then our embedding proceeds by first randomly partitioning $V(G)$ into two parts $V_1,V_2$ such that
\begin{eqnarray}\label{eq50}
|V_1|=|T'|+22\De m~\text{and}~|V_2|=(k'-1)r-22\De m. \nonumber
\end{eqnarray}
We can easily check that $|V_1|\ge \frac{n}{2},|V_2|\ge\frac{n_h}{16}-22\De m\ge\frac{n}{\sqrt{d}}$. Since $G$ is an $(n,d)$-expander and
\begin{eqnarray}\label{eq51}
d_i:=\frac{|V_i|}{5n}d\ge \frac{\sqrt{d}}{5}>\max\{2\De,\log^4 n\}~\text{for every}~ i\in [2],
\end{eqnarray}
by Lemma~\ref{randompart}, there exists a partition $V_1\cup V_2=V$ such that $G$ $d_i$-expands into $V_i$ for all $i\in[2]$.
Based on this,  we can first apply Theorem~\ref{almost} to $G[V_1]$ to find an embedding of $T'$ inside $V_1$. Indeed, it suffices to ensure that $d_1\ge 2 \De$ and the order of $T'$ is no larger than $|V_1| - 4\Delta \left\lceil\frac{|V_1|}{2d_1} \right\rceil$. This easily follows from \eqref{eq51} and $|V_1| -|T'|=22\De m\ge 4\De\left\lceil\frac{5n}{2d}\right\rceil=4\De\left\lceil\frac{|V_1|}{2d_1}\right\rceil.$

Let $\phi|_{T'}$ be the resulting embedding of $T'$ and $A,B$ be the set of images of ends $a_i,b_i$, respectively. Moreover, denote  by $V_1'$ the set of leftover vertices not covered by $\phi|_{T'}$. Then $|V_1'|=22\De m$. Now it remains to embed the bare paths $Q_i,i\in[r]$ in $V_2':=V_2\cup V_1'$. Recall that $h=\lceil\sqrt{\log n}\rceil$ and $k=\lceil\log^3 n\rceil$. Since $G$ $d_2$-expands into $V_2$ and $|V_2'|\le 2|V_2|$, one can easily obtain that $G$ actually $\frac{d_2}{2}$-expands into $V_2'$. As $d_2\ge \log^4 n$ and $k'=k-4h\ge 10^3\log^2 n$,
%\todo{need $k\ge \log^3 n$}
such a collection of vertex-disjoint paths can be obtained by Theorem~\ref{pathcover} applied to $G[A\cup B\cup V_2]$ with $(d_2,k'+1,(k'+1)r,V_2)$ in place of $(d,\ell,n, W)$, which together with $\phi|_{T'}$ completes the embedding of $T$.

\subsubsection{Many long caterpillars in $T$: \textbf{Case B}}\label{sec5.2}

% \begin{figure}[htb]
%\center{\includegraphics[width=7.4cm] {caterpillar-case.png}}
%\caption{}
%\label{fig3}
%\end{figure}

Without loss of generality, we may assume that the subpaths $Q_i\< P_i, 1\le i\le r/2$, are caterpillars in $T$ with at least one leg. Note that by taking induced subgraphs and renaming, we may further assume that each $Q_i$ has length $k':=\lceil\frac{k-4h}{2}\rceil$ and we can write $Q_i=a_0^ia_1^i\ldots a_{k'}^i$ so that $a_1^i$ is attached with at least one leg. Denote by $A_j, j=0,1,\ldots,k'$ the set of the vertices $a_j^i$ taken over all paths $Q_i, 1\le i\le r/2$. Moreover we write %$\mathcal{Q}:=\{Q_i: 1\le i\le \lfloor\frac{n_h}{16k}\rfloor\}$ and let
$L$ for the set of leaves of $T$ that are adjacent to vertices in $\bigcup_{j=1}^{k'-1} A_j$. Furthermore, we write $L^{+}:=N_T(L)$. Then $A_1\<L^{+}$ and $|L|\ge|L^{+}|\ge|A_1|= r/2$.

Let $T':=T-\bigcup_{j=1}^{k'-1} A_j\cup L$.
Now our embedding proceeds as follows.\\
\medskip
\begin{figure}[htb]
\center{\includegraphics[width=10.1cm] {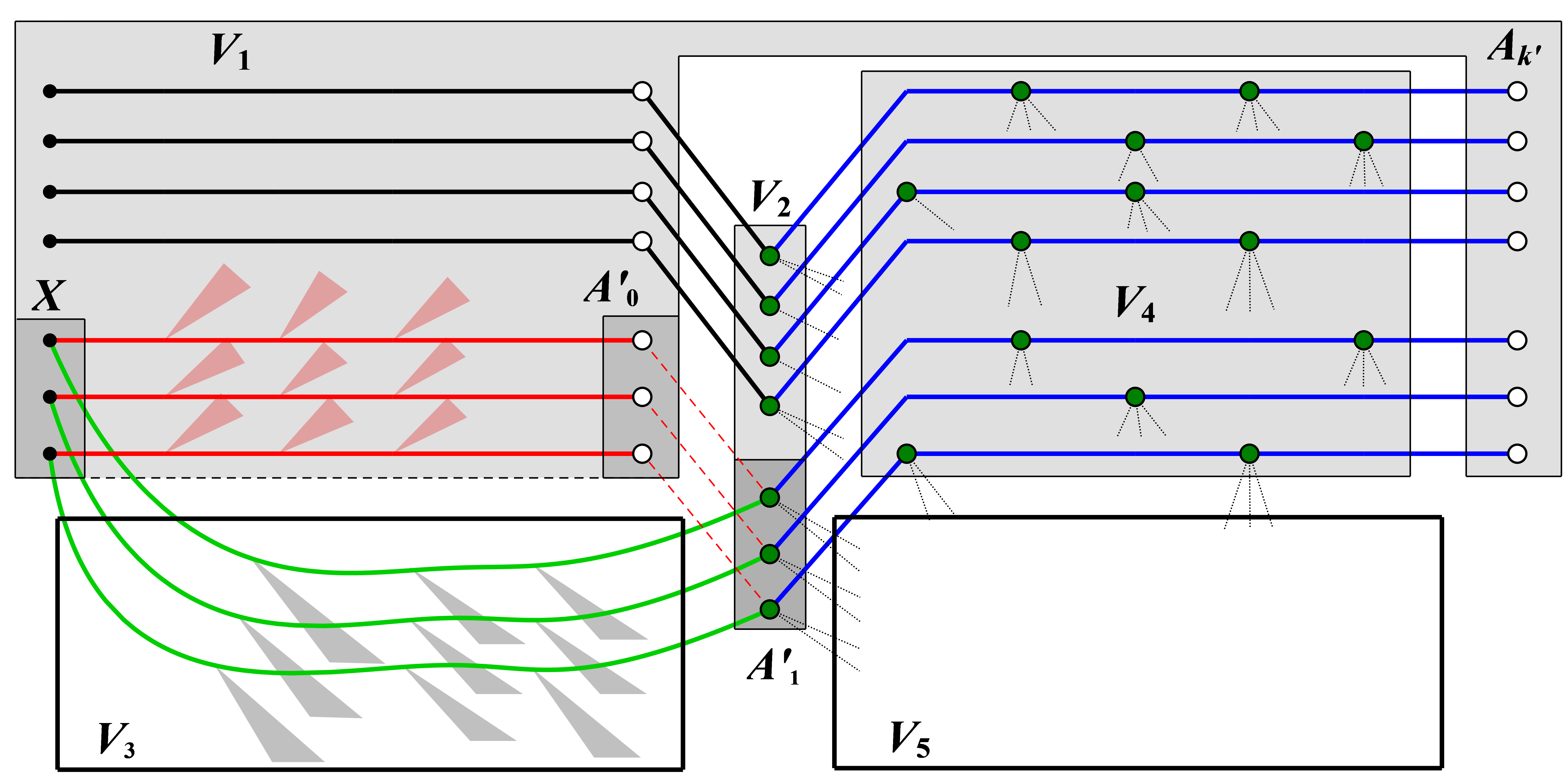}}
\caption{The red part is pruned and then used for building a star-matching}
\label{fig4}
\end{figure}

\noindent \textbf{Phase $0$. Partition $V(G)$}.
 \medskip\\
 We first randomly partition $V(G)$ into five parts $V_1,V_2,V_3,V_4,V_5$ (see Figure~\ref{fig4}) such that
\begin{eqnarray}\label{eq52}
|V_2|=|A_1|=r/2, ~|V_4|=(k'-2)|V_2|~\text{and}~|V_3|=|V_5|=\frac{|L|}{2}-11\De m.
\end{eqnarray}
By \eqref{eq2} and the fact that $|L|\le \De (|V_2|+|V_4|)$, we can easily check that
\begin{eqnarray}\label{eq53}
|V_2|=r/2\ge\lfloor\frac{n_{h}}{16k}\rfloor\ge\frac{n}{16k\De^{h+1}}\ge \frac{n}{\sqrt{d}} ,\nonumber \\
|V_3|=|V_5|=\frac{|L|}{2}-11\De m\ge\frac{1}{2}\lfloor\frac{n_h}{16k}\rfloor-11\De m\ge\frac{n}{\sqrt{d}},  \\
|V_1|=n-|V_2|-|V_3|-|V_4|-|V_5|=n-|T'|+22\De m>|A_0|=r/2\ge\frac{n}{\sqrt{d}}.\nonumber
\end{eqnarray}
Since $G$ is an $(n,d)$-expander  and
\begin{eqnarray}\label{eq54}
d_i:=\frac{|V_i|}{5n}d\ge \frac{\sqrt{d}}{5}>\max\{\De^{2\sqrt{\log n}},\log^4 n\}~\text{for every}~ i\in [5],
\end{eqnarray}
by Lemma~\ref{randompart}, there exists such a partition $\{V_1,V_2,V_3,V_4,V_5\}$ with the following property:
\stepcounter{propcounter}
\begin{enumerate}[label = ({\bfseries \Alph{propcounter}\arabic{enumi}})]
 %      \item\label{p3} For each $i\in [5]$ and $v\in V(G)$, $d(v,V_i)\ge \frac{|V_i|}{2n}d$.

       \item\label{p4} $G$ $d_i$-expands into $V_i$ for all $i\in[5]$.
\end{enumerate}

\medskip
\noindent \textbf{Phase $1$. Embed $T'$ in $G[V_1]$}.
\medskip\\
By property~\ref{p4}, we can apply Theorem~\ref{almost} to $G[V_1]$ to find an embedding $\phi|_{T'}$ inside $V_1$. Indeed, it suffices to ensure that $d_1\ge 2 \De$ and the order of $T'$ is no larger than $|V_1| - 4\Delta \lceil\frac{|V_1|}{2d_1}\rceil$. This easily follows from \eqref{eq54} and %$|T'|=n-|\bigcup_{j=1}^{k'-1} A_j|-|L|$
$|V_1| - |T'|=22\De m\ge 4\De\left\lceil\frac{|V_1|}{2d_1}\right\rceil.$

\medskip
\noindent \textbf{Phase $2$. Embed $A_1$ into $V_2$}.
 \medskip\\
We first build a maximal matching $M$ between $\phi(A_0)$ and $V_2$. Denote by $A'_0, A_1'$ the sets of vertices in $\phi(A_0)$ and $V_2$, respectively, that are not covered by $M$. Since $G$ is $m$-joined, we have \[|A'_0|=|A'_1|< m\le\frac{n}{2\De^{5h-5}}.\]
In order to complete the embedding of $A_1$, we shall use Corollary~\ref{rec} to build tree arrays.
For each $v\in A_0'$, let $x_v$ be the vertex such that $\phi^{-1}(x_v)$ is on $P_i\setminus Q_i$ with distance $h$ to $v$ and $X:=\{x_v: v\in A_0'\}$ (see Figure~\ref{fig4}). Let $\mathcal{I}$ be a family of $|A_1'|$ disjoint pairs obtained by arbitrarily pairing vertices between $X$ and $A_1'$.
Since \[|V_3|\ge \frac{n}{\sqrt{d}} ~ \text{and}~ |\mathcal{I}|< m,\] by property~\ref{p4} and Corollary~\ref{rec} with $s=h$ and $W=V_3$, we obtain a $(V_3,\mathcal{I},h,\Delta)$-tree array. Using the technique as in \ref{prun}\ref{rebuild}, one can obtain an embedding $\phi$ of $T-\bigcup_{j=2}^{k'-1} A_j-L$ such that $\phi(A_1)=V_2$.

\medskip
\noindent \textbf{Phase $3$. Embed $\bigcup_{j=2}^{k'-1} A_j$ into $V_4$}
\medskip\\
Let $\mathcal{P}$ be a family of the vertex pairs $\{\phi(a_1^i),\phi(a_{k'}^i)\}$, $1\le i\le r/2$. We shall use Theorem~\ref{pathcover} to construct for all pairs in $\mathcal{P}$, pairwise vertex-disjoint paths each of length $k'-1$ whilst covering all vertices in $V_4$.
Since $G$ $d_4$-expands into $V_4$ with $d_4>\log^4 n$ (see \eqref{eq54}) and $k'\ge 10^3\log^2n$, by applying Theorem~\ref{pathcover} to $G[V_4\cup V(\mathcal{P})]$ with $W=V_4$ and $\ell=k'$, we obtain such a family of paths. Let $\phi|_{T-L}$ be the resulting embedding and $R=V(G)\setminus \phi(T-L)$.

\medskip
\noindent \textbf{Phase $4$. Embed $L$ into $R$ via a star-matching}.
\medskip\\
Recall that $L^{+}$ is the neighborhood of $L$ in $T$. We shall build a star-matching between $\phi(L^{+})$ and $R$.
To be more precise, we define an auxiliary function $f:\phi(L^{+})\rightarrow [\De]$ by taking $f(v)=d_T(\phi^{-1}(v))-2$, that is, the number of leaves we need to attach to every $v\in \phi(L^{+})$ so as to complete the embedding.
Note that $V_5\<R,V_2\<\phi(L^{+})$. Then it follows from \ref{p4} that
for every $X\< R$ with $|X|\le m<\frac{5n}{2d}=\frac{|V_2|}{2d_2}$,
\[|N_G(X,\phi(L^{+}))|\ge|N_G(X,V_2)|\ge d_2|X|\ge \De|X|\] and similarly for every $Y\< \phi(L^{+})$ with $|Y|\le m$,
$|N_G(Y,R)|\ge |N_G(Y,V_5)|\ge d_5|Y|\ge \De|Y|$.
Applying Lemma~\ref{star1}, we have a desired $f$-matching, which together with $\phi|_{T-L}$ forms an embedding of $T$.

\subsubsection{Many subpaths $Q_i$ see pendant stars in $T$: \textbf{Case C}}

\begin{figure}[h]
\center{\includegraphics[width=10.1cm] {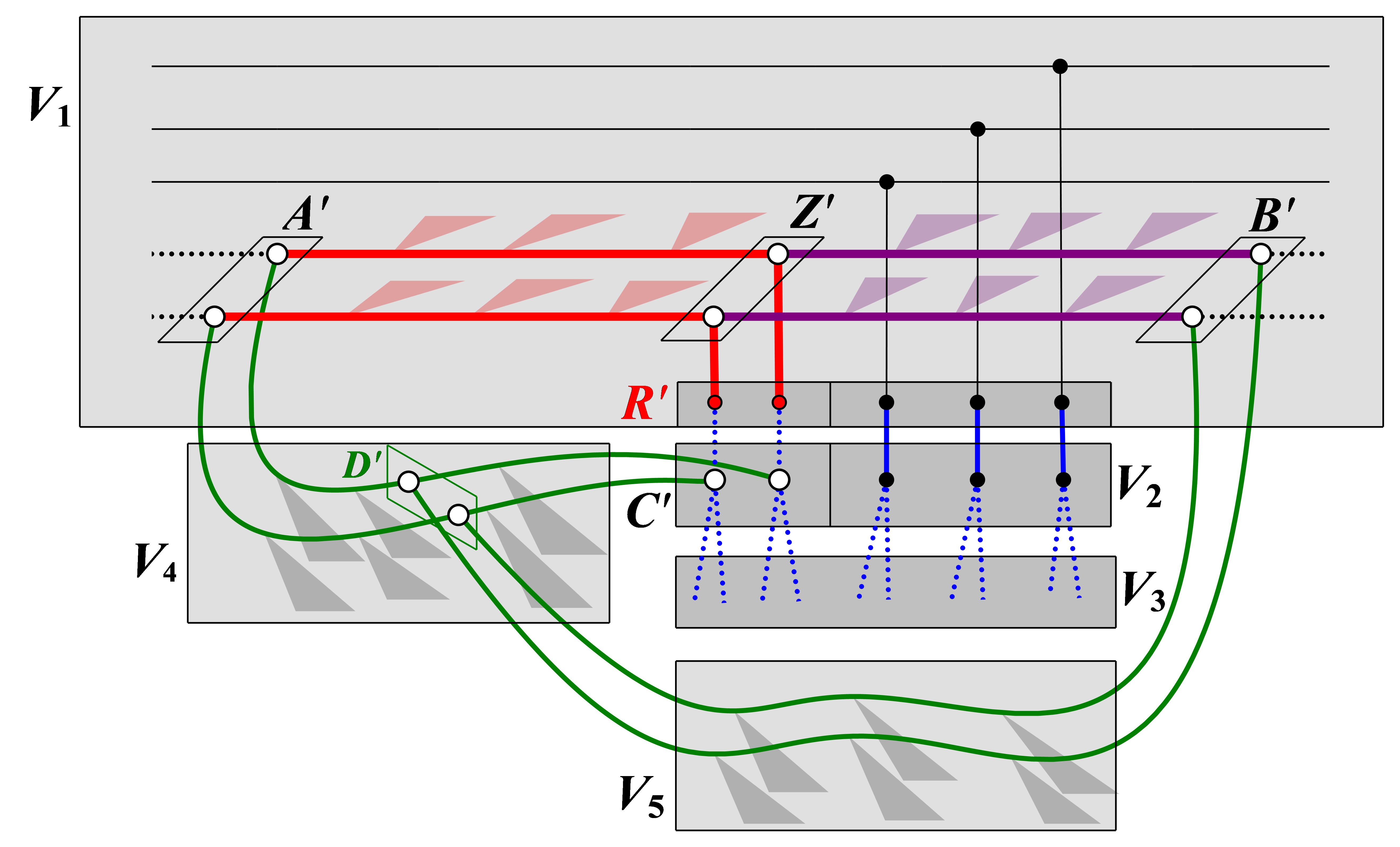}}
\caption{}
\label{fig5}
\end{figure}
Recall that $P_i$ is a bare paths in $T_h$ of length $k=\lceil\log^3n\rceil$ and has endpoints $x_i,y_i$ and $Q_i\<P_i$ is defined as a subpath of length $k-4h$, with each of ends of distance exactly $2h$ from $x_i$ and $y_i$, respectively.
Similarly, we may assume that for every $1\le i\le r/2$, the subpath $Q_i$ sees at least one pendant star, and we arbitrarily choose such a pendant star, denoted as $S_i$. Accordingly, we denote by
\begin{itemize}
  \item $C$ the set of centers $c_i$ of pendant star $S_i, i\in[r/2]$;
  \item $R$ the set of roots $r_i$ of pendant star $S_i, i\in[r/2]$;
  \item $Z$ the set of vertices $z_i\in V(Q_i), i\in[r/2]$ such that $Q_i$ sees $S_i$ from $z_i$;
  \item $A$ the set of vertices $a_i$ inside the subpath $P_i(x_i,z_i), i\in[r/2]$ such that each $a_i$ is of distance $h-1$ from $r_i$;
  \item $B$ the set of vertices $b_i$ inside the subpath $P_i(z_i,y_i), i\in[r/2]$ such that each $b_i$ is of distance $h$ from $z_i$;
\end{itemize}
Moreover we write $C^-$ for the set of leaves in $T$ that are attached to the vertices in $C$. Thus $r/2\le|C|\le |C^-|\le (\De-1) |C|$.
Let $T':=T-(C\cup C^-)$.
Now our embedding proceeds as follows.

\medskip
\noindent \textbf{Phase $0$. Partition $V(G)$ into $V_1,V_2,V_3,V_4,V_5$}.
\medskip\\
Similar to previous arguments, there exists a partition $V(G)=V_1\cup V_2\cup\cdots\cup V_5$ satisfying the following properties:
\stepcounter{propcounter}
\begin{enumerate}[label = ({\bfseries \Alph{propcounter}\arabic{enumi}})]
       \item\label{p5}   $|V_2|=|C|, ~|V_3|=|V_4|=|V_5|=\frac{|C^-|}{3}-7\De m, ~|V_1|=n-|C|-|C^-|+21\De m=|T'|+21\De m$;
%       \item\label{p6}  for each $i\in [5]$ and $v\in V(G)$, $d(v,V_i)\ge \frac{|V_i|}{2n}d$;

       \item\label{p7} $G$ $d_i$-expands into $V_i$ with $d_i:=\frac{|V_i|}{5n}d\ge \frac{\sqrt{d}}{5}>\max\{\De^{2\sqrt{\log n}},2\log n\}$, $i\in[5]$.
\end{enumerate}

\medskip
\noindent \textbf{Phase $1$. Embed $T'$ in $G[V_1]$}.
\medskip\\
By property~\ref{p7}, one can similarly apply Theorem~\ref{almost} to $G[V_1]$ to find an embedding $\phi|_{T'}$ inside $V_1$. Indeed, by Theorem~\ref{almost}, it suffices to ensure that $d_1\ge 2 \De$ and the order of $T'$ is at most $|V_1| - 4\Delta \lceil\frac{|V_1|}{2d_1}\rceil$, which easily follows since
$|V_1| -|T'|=21\De m >4\De\left\lceil\frac{|V_1|}{2d_1}\right\rceil.$

\medskip
\noindent \textbf{Phase $2$. Embed $C$ into $V_2$}.
\medskip\\
We first build a maximal matching $M$ between $\phi(R)$ and $V_2$. Denote by $R', C'$ the sets of vertices in $\phi(R)$ and $V_2$, respectively, that are not covered by $M$. Since $G$ is $m$-joined, we have $|R'|=|C'|< m$.
Using the same strategy as in Section~\ref{sec5.2}, we shall use Corollary~\ref{rec} (twice) to finish the embedding of $C$. Without loss of generality, we may write $R'=\{\phi(c_1),\ldots,\phi(c_{m-1})\}$, $B':=\{\phi(b_1),\ldots,\phi(b_{m-1})\}$ and $Z':=\{\phi(z_1),\ldots,\phi(z_{m-1})\}$. Let $\mathcal{I}$ be a family of $m-1$ disjoint pairs obtained by arbitrarily pairing vertices between $A':=\{\phi(a_1),\ldots,\phi(a_{m-1})\}$ and $C'$ (see Figure~\ref{fig5}).
Since \[|V_4|\ge \frac{n}{\sqrt{d}} ~ \text{and}~ |\mathcal{I}|< m,\] by property~\ref{p7} and applying Corollary~\ref{rec} with $s=h$ and $W=V_4$, we obtain a $(V_4,\mathcal{I},h,\Delta)$-tree array. Using the pruning-rebuilding technique as in \ref{prun}\ref{rebuild}, one can find an adjustment $\phi'$ of $\phi|_{T'}$ which maps every $(a_i,c_i)$-path in $T$ ($i\in[m-1]$) and the corresponding subtrees attached (the red parts in Figure~\ref{fig5}) into $V_4$, where the set $\{z_1,z_2,\ldots,z_{m-1}\}$ is instead mapped to a subset of $V_4$, say $D'$.  However, the current $\phi'$ may fail to map the other segments $P_i(z_i,b_i), i\in[m-1]$ (the purple parts) into paths in $G$. To overcome this, we use Corollary~\ref{rec} again to rebuild the paths $P_i(z_i,b_i), i\in[m-1]$ together with the corresponding subtrees attached. This is similarly done by applying Corollary~\ref{rec} with $\mathcal{I}=\{(\phi'(z_i),\phi'(b_i)): i\in[m-1]\}$ and $W=V_5$, and the resulting embedding, denoted by $\phi|_{T-C^-}$, manages to map $C$ into $V_2$.

\medskip
\noindent \textbf{Phase $3$. Embed $C^-$ via a star-matching}.
\medskip\\
Now it remains to embed $C^-$ into the remaining set of vertices in $V(G)$, say $L$. Note that $V_3\<L$ and by property~\ref{p7} it holds that for every $X\< V_2$ with $|X|\le m<\frac{5n}{2d}=\frac{|V_3|}{2d_3}$,
\[|N_G(X,L)|\ge|N_G(X,V_3)|\ge d_3|X|\ge \De|X|\] and similarly for every $Y\< L$ with $|Y|\le m$,
$|N_G(Y,V_2)|\ge d_2|Y|\ge \De|Y|$.
Lemma~\ref{star1} applied to $G$ with $(V_2,L)$ in place of $(A,B)$, gives a star-matching in which every $v\in V_2$ would be connected to $d_{T}(\phi^{-1}(v))-1$ vertices in $L$, and this combined with $\phi|_{T-C^-}$ completes an embedding of $T$.

\subsection{Tree array}
We first introduce useful tools for the proof of Lemma~\ref{rec1}.
Glebov, Johannsen and Krivelevich \cite{Glebov13} recently modified Haxell's method~\cite{Haxell01} to develop a very flexible approach for embedding bounded degree trees. We first need a key notion of $(d, m)$-extendable subgraph.

\begin{definition}\cite[$(d,m)$-extendable]{Glebov13}\label{extendable}
  Let $d,m\in \N$ satisfy $m\geq 1$ and $d\geq 3$, let $G$ be a graph, and let $S\< G$ be a subgraph. We say that $S$ is \emph{$(d,m)$-extendable} if $S$ has maximum degree at most $d$ and
\[
|\Gamma_G(X)\setminus V(S)|\geq(d-1)|X|-\sum_{x\in X\cap V(S)}(d(x,S)-1)
\]
for all sets $X\<V(G)$ with $|X|\leq 2m$.
\end{definition}
The following results of Montgomery~\cite{Montgomery19} give sufficient conditions under which a $(d,m)$-extendable subgraph can be extended yet remain $(d,m)$-extendable. This allows us to build up a collection of trees rooted at a given set of vertices.
%subgraph shows that under certain conditions a $(d,m)$-extendable subgraph $S$ can have a leaf added to any vertex with degree strictly less then $d$, so that the resulting subgraph remains $(d,m)$-extendable. This allows us, by repeated attachment of leaves, to build up from a single $(d,m)$-extendable vertex any chosen spanning tree.

\begin{lemma}\cite[Corollary 3.12]{Montgomery19}\label{extendpath}
Let $d, m, \ell\in \N$ satisfy $m \ge 1$ and $d \ge3$. Letting $k =\lceil\frac{\log(2m)}{\log(d -
1)}\rceil$, suppose $\ell \ge 2k+1$. Let $G$ be an $m$-joined graph and $S$ be a $(d, m)$-extendable subgraph
of $G$ with at most $|G| - 10dm - (\ell - 2k - 1)$ vertices.
Suppose $a$ and $b$ are two distinct vertices in $S$, both with degree at most $d/2$ in S.
Then there is an $(a, b)$-path P of length $\ell$ with internal vertices outside $S$, such that
$S\cup P$ is $(d, m)$-extendable.
\end{lemma}

\begin{lemma}\cite[Corollary 3.7]{Montgomery19}\label{extendtree}
Let $d,m\in\N$ satisfy $m \geq 1$, $d \geq 3$ and $T$ be a tree with maximum
degree at most $d-1$, which contains a vertex $t \in V (T)$. Let $G$ be an $m$-joined graph and
suppose $S$ is a $(d, m)$-extendable subgraph of $G$ with maximum degree $d$. Let $r\in V(S)$
and suppose \[|S|+|T| \le |G|-2dm-3m.\] Then, there is a copy $T'$ of $T$ in $G-(V(S)\setminus\{r\})$,
in which $t$ is mapped to $r$, such that $S \cup T'$ is $(d, m)$-extendable in $G$.
\end{lemma}

\begin{proof}[Proof of Lemma~\ref{rec1}]
Given integers $m,d_1,\De\in \mathbb{N}$ as in the statement, we fix $G$ to be an $m$-joined graph and $W\<V(G)$ such that $G$ $d_1$-expands into $W$ with $d_1\ge \De+2$.
Now we build the desired tree array in the following two steps. At the first step, we shall use Lemma~\ref{extendpath} to build pairwise vertex-disjoint paths with internal vertices in $W$, each connecting a prescribed pair of vertices in $\mathcal{I}$.
Let $S$ be an empty graph on vertex set $V(\mathcal{I})$.
\begin{claim}
  The subgraph $S$ is a $(d_1, m)$-extendable subgraph in $G'=G[W\cup V(\mathcal{I})]$.
\end{claim}
\begin{proof}\renewcommand*{\qedsymbol}{$\blacksquare$}
By Definition~\ref{extendable}, it is sufficient to show that for every set $X\subseteq V(G')$ with $|X| \le 2m$, we have \[|\Gamma_{G'}(X)\setminus V(\mathcal{I})| \ge d_1|X|.\] As $N_G(X,W)\<\Gamma_{G'}(X)\setminus V(\mathcal{I})$, it suffices to verify that $|N_G(X,W)| \ge d_1|X|$. This trivially follows since $G$ $d_1$-expands into $W$ and $|X|\le 2m<\frac{|W|}{2d_1}$.
\end{proof}
Therefore, by repeatedly applying Lemma~\ref{extendpath} to $G'$ with $(d_1,s,m)$ in place of $(d,\ell, m)$, we can connect the pairs of vertices from $\mathcal{I}$ by vertex-disjoint paths $P_1, \ldots,P_t$ each of length $s$ in $G'$.
This can be done as in this process the resulting subgraph $S_i:=S\cup P_1\cup\cdots\cup P_{i}$, $i\le t$, would still be a $(d_1, m)$-extendable graph of order at most
\[(s+1)t\le |W|-10d_1m-s.\]

For the second step, let $R:=V(S_t)\setminus V(\mathcal{I})$.
Then we shall repeatedly use Lemma~\ref{extendtree} to build for all vertices $v\in R$, pairwise vertex-disjoint $\Delta$-ary trees $T_v$ of height $s$, where $T_v$ has root $v$ and all other vertices in $W\setminus R$. In fact, in the initial step, we have that \[|S_t|+|T_v|\le (s+1)t+(1+\Delta+\Delta^2+\cdots+\Delta^{s})< (s+1)\De^{s+1}t<|W|-2d_1m-3m.\] For an arbitrary vertex $v\in R$, we apply Lemma~\ref{extendtree} to obtain a desired copy $T_v$ of $\Delta$-ary tree of height $s$ with a root $v$ and in particular we have that $S_t\cup T_v$ is also $(d_1, m)$-extendable in $G[W\cup V(\mathcal{I})]$. We repeat this step for every $v\in R$ iteratively such that the corresponding $\Delta$-ary trees $T_v$ are pairwise vertex disjoint. Note that this can be done as at the end of this process the resulting subgraph $\bigcup_{v\in R}T_v\cup S_t$ would still be $(d_1, m)$-extendable with maximum degree $\De+2\le d_1$ and order at most
%\todo{lower bound $|W|>m(s\De^s+2^{\frac{\log n}{s}})$}
\[
|S_t|+|R|(1+\Delta+\Delta^2+\cdots+\Delta^{s})\le t(s+1)\Delta^{s+1}\le |W|-2d_1m-3m.
\]
The family $\{T_v: v\in R\}$ of trees together with the family $\{P_i: i\in[t]\}$ of paths as above form a desired $(W,\mathcal{I},s,\Delta)$-tree array.
\end{proof}

%\section{Embedding spanning trees in $(n,d)$-expanders}
\section{Proof of Theorem~\ref{th2}}

The proof of Theorem~\ref{th2} is divided into two cases depending on whether $T$ has many pendant stars or many vertex-disjoint caterpillars of the same length.
This would accordingly be decoded by the following two results.

\begin{lemma}[Pendant stars]\label{pendant}
For any constant $0<\ga<\frac{1}{2}$, there exists $c>0$ such that the following holds for sufficiently large integer $n$. For all $\Delta\in \mathbb{N}$ with $\Delta \le c\sqrt{n}$, every
$(n, \frac{100}{\ga}\Delta \sqrt{n})$-expander is universal for the trees in $\mathcal{T}(n, \Delta)$ with  at least $\frac{\ga n}{\De}$ pendant stars.
\end{lemma}

\begin{lemma}[Caterpillars]\label{caterpillar}
For any integer $k\ge 800$ there exists $C>0$ such that the following holds for sufficiently large integer $n$. For all $\Delta\in \mathbb{N}$ with $\Delta \le\frac{\sqrt{n}}{2C}$, every
$(n, C\Delta \sqrt{n})$-expander is universal for the trees in $\mathcal{T}(n, \Delta)$ with at least $\frac{ n}{4k\De}$ vertex-disjoint caterpillars each of length $k$.
\end{lemma}

\begin{proof}[Proof of Theorem~\ref{th2}]
Take $k=800$, and choose $\frac{1}{n}\ll c\ll\frac{1}{C}\ll\frac{1}{k}$. Let $G$ be an $(n,d)$-expander with $d\ge C\De\sqrt{n}$, where $\De\le c\sqrt{n}$. By Corollary~\ref{newdiv}, every $T\in \mathcal{T}(n, \Delta)$ has either at least $\frac{n}{4k\De}$ pendant stars or $\frac{n}{4k\De}$ vertex-disjoint caterpillars each of length $k$. For the former case, by applying Lemma~\ref{pendant} with $\ga=\frac{1}{4k}$, one can find a copy of $T$ in $G$. For the latter case, as $c\ll \frac{1}{C}$, Lemma~\ref{caterpillar} immediately gives a copy of $T$ in $G$.
\end{proof}
\subsection{Many pendant stars: proof of Lemma~\ref{pendant}}
Choose $\frac{1}{n}\ll c\ll\ga$ and let $T$ be an $n$-vertex tree with maximum degree $\De\le c\sqrt{n}$. Let $S_1,S_2,\ldots,S_{\frac{\ga n}{\De}}$ be a collection of pendant stars in $T$, and $A$ be the set of roots of the pendant stars in the collection, where for every $v\in A$, we denote by $s(v)$ the number of pendant stars in the collection that are rooted at $v$. Denote by $B=\{b_1,b_2,\ldots,b_{\frac{\ga n}{\De}}\}$ the set of centers of the pendant stars in the collection and by $C$ the set of leaves attached to the vertices in $B$. Thus $\sum_{v\in A}s(v)=|B|$ and $|B|\le|C|<\De |B|$. Let $T':=T-(B\cup C)$.

Recall that $G$ is an $(n,d)$-expander with $d=\frac{100}{\ga}\Delta \sqrt{n}$ and we write $m:=\frac{n}{2d}$. Then our embedding proceeds as follows.

 \medskip
\noindent \textbf{Phase $0$. Partition $V(G)$}.
\medskip\\
We first randomly partition $V(G)$ into four parts $V_1,V_2,V_3,V_4$ such that
\begin{eqnarray}\label{eq6}
|V_1|=n-|B|-|C|+22\De m, ~|V_2|=|B|=\frac{\ga n}{\De}~\text{and}~ |V_3|=|V_4|=\frac{|C|}{2}-11\De m.
\end{eqnarray}
One can easily check that $|V_i|\ge\frac{\ga n}{2\De}-11\De m=\frac{\ga n}{2\De}-\frac{11\sqrt{n}}{200}\ge\frac{\ga n}{3\De}$ as $c\ll \ga$.
Since $G$ is an $(n,d)$-expander with
\begin{eqnarray}\label{eq61}
d=\frac{100}{\ga}\Delta \sqrt{n}~\text{and thus}~d_i:=\frac{|V_i|}{5n}d\ge 6\sqrt{n}>\max\{2\De,2\log n\}~\text{for every}~ i\in [4],
\end{eqnarray}
%\todo{need $d\ge \De \log n$}
by Lemma~\ref{randompart}, there exists such a partition with the following property:
\stepcounter{propcounter}
\begin{enumerate}[label = ({\bfseries \Alph{propcounter}\arabic{enumi}})]
 %      \item\label{p8} For each $i\in [4]$ and $v\in V(G)$, $d(v,V_i)\ge \frac{|V_i|}{2n}d$.

       \item\label{p9} $G$ $d_i$-expands into $V_i$ for all $i\in[4]$.
\end{enumerate}

\medskip
\noindent \textbf{Phase $1$. Embed $T'$ in $G[V_1]$}.
\medskip\\
By property~\ref{p9}, we can apply Theorem~\ref{almost} to $G[V_1]$ to find a copy of $T'$ inside $V_1$. Indeed, as $d_1\ge 2 \De$ in \eqref{eq61},  it suffices to ensure that the order of $T'$ is no larger than $|V_1| - 4\Delta \lceil\frac{|V_1|}{2d_1}\rceil$, and this easily follows as $|V_1| - |T'|=22\De m >4\De\left\lceil\frac{|V_1|}{2d_1}\right\rceil$.
The resulting embedding of $T'$ is denoted as $\phi|_{T'}$ and
write $V_1'=V_1\setminus\phi(V(T'))$ for the set of leftover vertices in $V_1$. Then we shall extend $\phi$ as follows. %(hereinafter, we may use $\phi$ to denote the up-to-date embedding).

\medskip
\noindent \textbf{Phase $2$. Embed $B$ into $V_2\cup V_3$}.
\medskip\\
We first find a maximal star-matching $M_1$ where every star has a center $\phi(v)$ with $v\in A$ and exactly $s(v)$ leaves in $V_2$. Let $A_1=\phi(A)\cap V(M_1), B_1=V_2\cap V(M_1)$ and $A_2=\phi(A)\setminus A_1, V_2'=V_2\setminus B_1$.
\begin{claim}
 $|A_2|< m~\text{and thus}~|V_2'|\le \De m$.
\end{claim}
Indeed, suppose $|A_2|\ge m$ and thus $|V_2'|=\sum_{v\in \phi^{-1}(A_2)}s(v)=|A_2|+k$ for some $k\ge 0$. Then as $G$ is $m$-joined and $|A_2|\ge m$, we have $\sum_{v\in \phi^{-1}(A_2)}d_G(\phi(v),V_2')=e(A_2,V_2')\ge k+1$. Since $M_1$ is maximal, we have $d_G(\phi(v), V_2')\le s(v)-1$ for every $v\in\phi^{-1}(A_2)$.
Thus, $\sum_{v\in \phi^{-1}(A_2)} d_G(\phi(v), V_2')\le \sum_{v\in \phi^{-1}(A_2)}(s(v)-1)=|V_2'|-|A_2|=k$, a contradiction.

Next we shall greedily build a star-matching $M_2$ between $A_2$ and $V_3$, where every star has a center $\phi(v)\in A_2$ and exactly $s(v)$ leaves in $V_3$. By combining the Hall's condition (see \eqref{eq4}), the existence of such a star-matching easily follows since $G$ $d_3$-expands into $V_3$ with $d_3>\De$ by \eqref{eq61}.
%\todo{need $d\ge \De^2$}.
Let $B_2=V_3\cap V(M_2)$ and $V_3'=V_3\setminus B_2$.

\medskip
\noindent \textbf{Phase $3$. Embed $C$ into $R:=V_1'\cup V_2'\cup V_3'\cup V_4$ via a star-matching}
\medskip\\
It remains to build a star-matching between $B_1\cup B_2$ and $R$. To be more precise, we define an auxiliary function $f: B_1\cup B_2\rightarrow [\De]$ by taking $f(v)=d_T(\phi^{-1}(v))-1$, that is, the number of leaves we need to attach to every $v\in B_1\cup B_2$ so as to complete the embedding.
By property~\ref{p9} and \eqref{eq61}, we observe that $G$ $d_4$-expands into $V_4$ for $d_4>\De$
%\todo{need $d\ge \De^2$}
and every $v\in R$ has
\[d(v,B_1\cup B_2)\ge d(v,V_2)-|V_2'|\ge \frac{|V_2|}{5n}d-\De m\overset{\eqref{eq6}}{\ge}\frac{\ga d}{5\De}-\De\frac{n}{2d}\overset{\eqref{eq61}}{>}\frac{n}{2d}=m.\]%\todo{$d^2>\De^2n$} \\%where the last inequality holds since $\frac{\ga d}{2\De}>\De\frac{n}{d}$.
Since $G$ is $m$-joined, we can apply Lemma~\ref{star} to $G$ with $(B_1\cup B_2,R,d_4)$ in place of $(A,B,d)$ to obtain a desired $f$-matching, which completes an embedding of $T$.

\subsection{Many caterpillars: proof of Lemma~\ref{caterpillar}}
In this section, we consider the case when $T$ has $\frac{n}{4k\De}$ vertex-disjoint caterpillars of a constant length $k$.

\begin{proof}[Proof of Lemma~\ref{caterpillar}]
Choose $\frac{1}{n}\ll\frac{1}{C}\ll\frac{1}{k}$
 and throughout the proof we have \[m=\frac{n}{2d},~d=C\De\sqrt{n}~\text{and}~\De\le \frac{\sqrt{n}}{2C}.\]
Assume that an $n$-vertex tree $T$ contains a collection of $\frac{n}{4k\De}$ vertex-disjoint caterpillars each containing a central path $P_i, 1\le i\le\frac{n}{4k\De}$, of length $k\ge 800$, where $x_i,y_i$ are the ends of $P_i$. %, for every $i\in\lfloor\frac{n_h}{4k}\rfloor$.
In the forthcoming proofs, we shall write $r=\frac{n}{8k\De}$ and consider two subcases as follows:
\begin{enumerate}
  \item [(1)] at least $r$ paths $P_i$ are bare in $T$;
  \item [(2)] at least $r$ paths $P_i$ are internally attached with leaves in $T$;
 % \item [(3)] at least one quarter of the paths $Q_i$ see pendant stars in $T$.
\end{enumerate}
 For the first case, we may assume that $P_1,P_2,\ldots,P_{r}$ are bare paths in $T$ and let $T'=T-\cup_{i\in[r]}V(P_i-\{x_i,y_i\})$. Then our embedding proceeds by first randomly partitioning $V(G)$ into two parts $V_1,V_2$ such that
\begin{eqnarray}\label{eq62}
|V_1|=|T'|+21\De m~\text{and}~|V_2|=(k-1)r-21\De m.
\end{eqnarray}
We can easily check that $|V_1|\ge \frac{n}{2},|V_2|\ge\frac{n}{16\De}-21\De m\ge\frac{n}{32\De}$ as $d\ge C\De^2$. Note that $G$ is an $(n,d)$-expander and
\begin{eqnarray}\label{eq63}
d_i:=\frac{|V_i|}{5n}d\ge\frac{d}{160\De}=\frac{C}{160}\sqrt{n}\ge\max\{2\De,2\log n\}~\text{ for every }~i\in [2].
\end{eqnarray}
By Lemma~\ref{randompart}, there exists a partition $V_1\cup V_2=V$ such that $G$ $d_i$-expands into $V_i$ for all $i\in[2]$.
Based on this,  we can first apply Theorem~\ref{almost} to $G[V_1]$ to find an embedding of $T'$ inside $V_1$. Indeed, as $d_1\ge 2 \De$, it suffices to ensure that the order of $T'$ is no larger than $|V_1| - 4\Delta \lceil\frac{|V_1|}{2d_1}\rceil$. This easily follows since $|V_1| - |T'|=21\De m >4\De\left\lceil\frac{|V_1|}{2d_1}\right\rceil.$
Let $\phi$ be the resulting embedding of $T'$ and $X,Y$ be the set of images of ends $x_i,y_i$, respectively. Now it remains to embed the bare paths $P_i,i\in[r]$ using $V_2$. By the choice of $\frac{1}{C}\ll\frac{1}{k}$ and the fact that $d_2\ge\frac{C}{160}\sqrt{n}\ge C_{\ref{pathcover2}}k\sqrt{n}$, such a collection of vertex-disjoint paths can be obtained by Lemma~\ref{pathcover2} applied to $G[X\cup Y\cup V_2]$ with $(d_2,k+1,(k+1)r,V_2)$ in place of $(d,\ell,n,W)$, which together with $\phi$ completes the embedding of $T$.

For the second case, the proof strategy is similar to that in Section~\ref{sec5.2}. Without loss of generality, we may assume that $P_1,P_2,\ldots, P_{r}$ are central paths from vertex-disjoint caterpillars in $T$, each with at least one leg. Following the technique in Section~\ref{sec5.2}, we take subpaths $Q_i\< P_i, 1\le i\le r$ such that each $Q_i$ has length $k'=\lfloor k/2\rfloor$ and one can write $Q_i=a^i_0a^i_1\ldots a^i_{k'}$ so that $a^i_1$ is attached with at least one leaf. For every $j\in\{0,1,\ldots,k'\}$, denote by $A_j$ the set of the vertices $a^i_j$ taken over all paths $Q_i, 1\le i\le r$.  Moreover we write %$\mathcal{Q}:=\{Q_i: 1\le i\le \lfloor\frac{n_h}{16k}\rfloor\}$ and let
$L$ for the set of all leaves in $T$ that are attached to the internal vertices of the paths $Q_i,1\le i\le r$. Write $L^+=N_T(L)$. Then $A_1\<L^+$ and $r=|A_j|\le|L|$ for every $j\in\{0,1,\ldots,k'\}$.

Let $T':=T-(\bigcup_{j=1}^{k'-1} A_j\cup L)$.
Now our embedding proceeds as follows.
\medskip\\
\noindent \textbf{Phase $0$. Partition $V(G)$}.
 \medskip\\
 We first randomly partition $V(G)$ into four parts $V_1,V_2,V_3,V_4$ such that
\begin{eqnarray}\label{eq55}
|V_1|=|T'|+21\De m,~|V_2|=|A_1|=r, |V_3|=(k'-2)r~\text{and}~|V_4|=|L|-21\De m.\nonumber
\end{eqnarray}
Observe that for every $i\in[4]$, $|V_i|\ge r-21\De m\ge\frac{n}{16k\De}$ as $d\ge C\De^2$.
Since $G$ is an $(n,d)$-expander and
\begin{eqnarray}\label{eq56}
d_i:=\frac{|V_i|}{5n}d\ge \frac{d}
{80k\De}=\frac{C}{80k}\sqrt{n}>\max\{2m,2\De,2\log n\}~\text{ for every}~ i\in [4],
\end{eqnarray} by Lemma~\ref{randompart}, we can obtain such a partition with the following property:
\stepcounter{propcounter}
\begin{enumerate}[label = ({\bfseries \Alph{propcounter}\arabic{enumi}})]
       \item\label{p44} $G$ $d_i$-expands into $V_i$ for all $i\in[4]$.
\end{enumerate}

\medskip
\noindent \textbf{Phase $1$. Embed $T'$ in $G[V_1]$}.
\medskip\\
By property~\ref{p44}, we can apply Theorem~\ref{almost} to $G[V_1]$ to find a copy of $T'$ inside $V_1$, denoted as $\phi|_{T'}$. Indeed, this easily follows since $d_1\ge 2 \De$ in \eqref{eq56} and $|V_1| -|T'|= 21\De m> 4\De\left\lceil\frac{|V_1|}{2d_1}\right\rceil.$

\medskip
\noindent \textbf{Phase $2$. Embed $A_1$ into $V_2\cup V_3$}.
 \medskip\\
We first build a maximal matching $M_1$ between $\phi(A_0)$ and $V_2$. Denote by $A'_0, A_1'$ the sets of vertices in $\phi(A_0)$ and $V_2$, respectively, that are not covered by $M$. Since $G$ is $m$-joined, we have $|A'_0|=|A'_1|< m$.
Moreover, by property~\ref{p44} and \eqref{eq56}, we have that $d(v,V_3)\ge d_3\ge 2m$. Thus one can easily obtain a matching $M_2$ from $A_0'$ to a subset $B_1\<V_3$ with $|B_1|=|A_0'|$. Let $V_2'=(V_2\setminus A_1')\cup B_1$ and $V_3'=(V_3\setminus B_1)\cup A_0'$. Then $|V_3'\setminus V_3|=|V_2'\setminus V_2|\le m$ and note that $V_2'=\phi(A_1)$. Furthermore, it is easy to verify that $G$ $\frac{d_2}{2}$-expands into $V_2'$ and $\frac{d_3}{2}$-expands into $V_3'$ as $d_i\ge 2m$.
%\todo{We need $C_k$ on d}.% as all but at most $m$ vertices in $A_1$ are mapped into $V_2$.

\medskip
\noindent \textbf{Phase $3$. Embed the subpaths $Q_i-a_0^i$ using $V_3'$}
\medskip\\
By the choice of $\frac{1}{C}\ll\frac{1}{k}$ and the fact that $\frac{d_3}{2}\ge\frac{C}{160k}\sqrt{n}\ge C_{\ref{pathcover2}}k\sqrt{n}$, such a collection of vertex-disjoint paths can be obtained by Lemma~\ref{pathcover2} applied to $G[V_2'\cup \phi(A_{k'})\cup V_3']$ with $(\frac{d_3}{2},k'+1,(k'+1)r,V_3')$ in place of $(d,\ell,n,W)$. This yields an embedding $\phi|_{T-L}$.

Hence it remains to embed $L$ into the remaining set $R$ of vertices with $R=V(G)-\phi(T-L)$.

\medskip
\noindent \textbf{Phase $4$. Embed $L$ into the leftovers via a star-matching}.
\medskip\\
Recall that $L^+$ is the neighborhood of $L$ in $T$ and $A_1\<L^+$. Also observe that $V_4\<R$. We shall build a star-matching between $\phi(L^+)$ and $R$. To finish the embedding of $T$, we define an auxiliary function $f:\phi(L^+)\rightarrow [\De]$ by taking $f(v)=d_T(\phi^{-1}(v))-2$, that is, the number of leaves we need to attach to every $v\in \phi(L^+)$.
Then as $G$ $\frac{d_2}{2}$-expands into $V_2'$, we observe that for every $v\in R$
\[d(v,\phi(L^+))\ge d(v,V_2')\ge\frac{d_2}{2}\overset{\eqref{eq56}}{\ge} m.\]
Moreover, as $V_4\<R$ and $G$ $d_4$-expands into $V_4$ with $d_4\overset{\eqref{eq56}}{>}\De$, we obtain that
\[|N(X,R)|\ge |N(X,V_4)|\ge d_4|X|\ge \De|X|~\text{ for every}~X\<\phi(L^+)~\text{with}~ |X|\le m.\]
Applying Lemma~\ref{star} with $(\phi(L^+),R,\De)$ in place of $(A,B,d)$, we have a desired $f$-matching, which together with $\phi|_{T-L}$ forms an embedding of $T$.

\end{proof}
\section{Covering Expanders with paths}

This section is devoted to the proof of Lemma~\ref{pathcover2}. Before that, we first give a result as follows.

%\begin{lemma}\cite{}\label{connect}
%Let $m,n\in \N$ satisfy $m\leq n/800$, let $d=n/200m$ and let $n$ be sufficiently large. Let the graph $G$ with $n$ vertices have the property that any set $A\<V(G)$ with $|A|=m$ satifies $|N(A)|\geq (1-1/64)n$. Suppose $G$ contains disjoint vertex sets $X$, $Y$ and $U$, with $X=\{x_1,\ldots,x_{2m}\}$, $Y=\{y_1,\ldots, y_{2m}\}$ and $|U|=\lceil n/8\rceil$. Then for all integers $k_i$, $i\in[2m]$, satisfying $2\log m/\log d\leq k_i\leq \frac{|U|}{2}$. Then, for some $j$, there is an $x_j,y_j$-path of length $k_j$ whose internal vertices lie in $U$.
%\end{lemma}
\begin{lemma}\label{connect}
Let $m,d_1\in \N$ with $m\ge d_1\ge 2$ and $G$ be a graph that is $m$-joined. Suppose $G$ contains disjoint vertex sets $X$, $Y$ and $U$, with $X=\{x_1,\ldots,x_{2m}\}$, $Y=\{y_1,\ldots, y_{2m}\}$ and $|U|\ge20d_1m$. Then for all integers $k_i$, $i\in[2m]$, satisfying \[\frac{2\log m}{\log d_1}+1\leq k_i\leq \frac{|U|}{2},\] there exists for some $i$ an $(x_i,y_i)$-path of length exactly $k_i$, whose internal vertices lie in $U$.
\end{lemma}

The proof of Lemma~\ref{connect} can be found in the appendix.

\lem\cite[Bipartite-template]{Montgomery14}\label{flexiblematching}
There is a constant $n_0$, such that for every $n\geq n_0$ with $3|n$, there exists a bipartite graph $H$ on vertex classes $X$ and $Y\cup Z$ with $|X|=n$, $|Y|=|Z|=2n/3$, and maximum degree $40$, so that the following is true. Given any subset $Z'\< Z$ with $|Z'|=n/3$, there is a matching between $X$ and $Y\cup Z'$.
\ma
\subsection{Proof of Theorem~\ref{pathcover2}}
Given a vertex $v\in V(G)$ and an integer $k\in \N$, a $k$-\emph{fan rooted at} $v$ is a subgraph consisting of $k$ triangles which mutually intersect on no other vertices than $v$, where every edge not incident with $v$ is called an \emph{absorber} for $v$. Now we are ready to prove Lemma~\ref{pathcover2}.
\pr
Choose $\frac{1}{n}\ll\frac{1}{C}\ll\frac{1}{\ell}$ and let $G, W$ be given such that $G$ $d$-expands into $W$, where
\begin{eqnarray}\label{eq70}
d=C\ell\sqrt{n},~|W|=\frac{\ell-2}{\ell}n~\text{and}~m=\lceil\frac{|W|}{2d}\rceil<\frac{\sqrt{n}}{C\ell}.
\end{eqnarray}
Then by Definition~\ref{expands} $G$ is $m$-joined.
Let \[X=\{x_1,x_2,\ldots,x_{n/\ell}\}, ~Y=\{y_1,y_2,\ldots,y_{n/\ell}\}~\text{and}~ r=\frac{n}{10^4\ell}.\]
Our goal is to find pairwise vertex-disjoint $(x_i,y_i)$-paths, $i\in[n/\ell]$, each of length $\ell$. Our proof proceeds in the following steps. %, which indeed cover the set $V(G)$.

We first randomly partition $W$ into four parts $W_1,W_2,W_3,W_4$ such that
\begin{eqnarray}\label{eq71}
|W_1|=|W_2|=2r~\text{and}~|W_i|= \frac{|W|-4r}{2}\ge \frac{n}{3}~\text{for every}~i\in\{3,4\}.
\end{eqnarray}
Let $d_i=\frac{|W_i|}{5|W|}d$ for every $i\in [4]$. Then $d_i>\frac{2r}{5n}d>2\log n$.
As $G$ $d$-expands into $W$, by Lemma~\ref{randompart}, there exists a partition $\{W_1,\ldots,W_4\}$ so that $G$ $d_i$-expands into $W_i$ for every $i\in[4]$.

\medskip
\noindent \textbf{Phase $1$. Building an absorbing structure}.
\medskip\\
Take subsets $X_1=\{x_1,x_2,\ldots,x_{3r}\},Y_1=\{y_1,y_2,\ldots,y_{3r}\}$ and write $X_2=X\setminus X_1, Y_2=Y\setminus Y_1$.
\begin{claim}\label{absorbmany} There is a subset $A\subseteq W_2\cup W_3\cup W_4$ with $|A|=3r(\ell-2)-r$ such that for any subset $U\< W_1$ with $|U|=r$, there is a collection of $3r$ vertex disjoint $(x_i,y_i)$-paths of length $\ell-1$, $i\in[3r]$, in $A\cup U$. In fact, therefore, such paths cover the set $A\cup U$.
\end{claim}
\begin{proof}[Proof of claim]\renewcommand*{\qedsymbol}{$\blacksquare$}
%Using Lemma \ref{splitexpand}, with $n$ sufficiently large, partition $W$ as $W_1\cup W_2\cup W_3$, such that, for each $i$, $|W_i|=|W|/3$ and $G$ $(20\log^2 n)$-expands into $W_i$.
We shall first create pairwise disjoint $40$-fans $F_v$ with root $v$ for every $v\in W_1\cup W_2$ and the remaining vertices inside $W_3$. To achieve this, we choose a subset $T\<W_3$ such that $|T|=\frac{|W_3|}{2}$ and $d(v,T)\ge \frac{1}{4}d(v,W_3)$ for every $v\in V(G)$. Such a set $T$ can be found by choosing $\frac{|W_3|}{2}$ vertices in $W_3$ uniformly at random and by applying Chernoff's inequality and a union bound. Let $T'=W_3\setminus T$.
We claim that for every set $L$ of $m$ vertices in $W_1\cup W_2$ and every set $S\<W_3$ of $\frac{n}{100}$ vertices, there exists a $40$-fan with a root in $L$ and the remaining vertices in $S$. In fact, as $G$ is $m$-joined, by averaging, there is a vertex, say $v\in L$, that has a set $S_v$ of neighbors inside $S$ with \[ |S_v|>\frac{|S|-m}{m}\ge\frac{n}{200m}\overset{\eqref{eq70}}{\ge} 80m.\] Thus as $G$ is $m$-joined, we can greedily pick a matching of $40$ edges in $S_v$, which together with $v$ form a $40$-fan $F_v$ as desired. Based on this claim we can greedily choose pairwise disjoint $40$-fans $F_v$ for all but at most $m$ vertices $v\in W_1\cup W_2$, where all the remaining vertices of such fans $F_v$ come from $T'$. This is possible as $|T'|=\frac{|W_3|}{2}>\frac{n}{6}\ge 80|W_1\cup W_2|+\frac{n}{100}$. Now it remains to find fans for a set $L'$ of at most $m$ vertices from $W_1\cup W_2$, and here we shall use $T$ (reserved to cover the leftover vertices). This can be trivially done as for every $v\in V(G)$, $d(v,T)\ge \frac{1}{4}d(v,W_3)\ge\frac{|W_3|}{20|W|}d>\frac{d}{60}\ge 160m\ge 80|L'|+80 m$. Denote by $A_v$ the set of absorbers for $v$ in $F_v, v\in W_1\cup W_2$ and let $M=\cup_{v\in W_1\cup W_2}A_v$. Then $M$ is indeed a matching of size $160r$ in $W_3$. Let $W_3'=W_3\setminus V(M)$ and thus $|W_3'|=|W_3|-320r$.

To describe how to route our paths through these triangles in the fans, we refer to an auxiliary bipartite graph $H$ obtained by Lemma \ref{flexiblematching} applied with $n=3r$. The bipartite graph $H$ has maximum degree $40$ and vertex classes $X=[3r]$ and $Y\cup Z$ such that $|Y|=|Z|=2r$ and for any subset $Z'\<Z$ with $|Z'|=r$ there is a perfect matching between $X$ and $Z'\cup Y$ in $H$. Fix an arbitrary bijection $\tau:Y\cup Z\to W_1\cup W_2$ with $\tau(Z)=W_1$. Now we first pick pairwise disjoint matchings $M_i\<M,i\in[3r]$ such that
\stepcounter{propcounter}
\begin{enumerate}[label = ({\bfseries \Alph{propcounter}\arabic{enumi}})]
       \item\label{p20} for every $i\in[3r]$ and every $v\in \tau(N_H(i))$, $M_i$ consists of exactly one edge in $A_v$.
\end{enumerate}
Recall that $X_1=\{x_1,\ldots,x_{3r}\}, Y_1=\{y_1,\ldots,y_{3r}\}$. Next, we shall build pairwise disjoint $(x_i,y_i)$-paths $P_i$ of length $\ell-2$ for every $i\in[3r]$ such that $M_i\<E(P_i)$ and the other internal vertices come from $W_3'\cup W_4$.
As $|M_i|\le 40$, we may take $|M_i|=40$ for instance and set $M_i=\{u_jv_j: j\in[40]\}$ to ease the notation. Observe that $\ell\ge 200> 3(|M_i|+1)+|M_i|$, every such $P_i$ ought to be constructed by connecting the corresponding $41$ pairs of vertices \[\{x_i,u_1\},\{v_1,u_2\},\ldots\{v_j,u_{j+1}\},\ldots,\{v_{40},y_i\}\] using pairwise disjoint paths $P_{i,j}$ of length $k_{i,j}\ge 3,j\in[41]$ satisfying
\[\sum_{j\in[41]}k_{i,j}=\ell-2-|M_i|.\]
To do this, in total we have a set of $\sum_{i\in[3r]}(|M_i|+1)\le 123r$ such pairs to connect as above, say $\mathcal{P}$, and the number of vertices used in all these connections is at most $3r\ell$. Note that by the choice of $\frac{1}{C}\ll\frac{1}{\ell}$ and $r=\frac{n}{10^4\ell}$, %\todo{lower bound $r=O(\frac{n}{\ell})$.}
we have
\[
|W_3'|=|W_3|-320r=\frac{|W|}{2}-322r\overset{\eqref{eq71}}{>}\frac{n}{6}> 3r\ell+20m^2.
\]
By repeatedly applying Lemma~\ref{connect} with $d_1=m$ and $U$ playing the role of the set of vertices in $W_3'$ uncovered by previous connections, one can greedily connect all but a subfamily $\mathcal{P}'\<\mathcal{P}$ of at most $2m$ pairs via vertex-disjoint paths $P_{i,j}$ with all the internal vertices in $W_3'$. Now we shall connect the remaining pairs in $\mathcal{P}'$ using $W_4$. Note that $G$ $d_4$-expands into $W_4$ and thus for every $v\in V(G)$, \[d(v,W_4)\ge d_4=\frac{|W_4|}{5|W|}d> \frac{d}{20}\overset{\eqref{eq70}}{>}10 m\ell.\]
Thus one can greedily complete the connection for $\mathcal{P}'$ as the total number of vertices used in connections is at most $2m\ell$.
%in the process every pairs in $\mathcal{P}'$ with a path of length $k$ for every $k\in[3,\ell]$.

Let $A=\cup_{i\in[3r]}(P_i\setminus\{x_i,y_i\})\cup W_2$ and thus $|A|=3r(\ell-2)-r$. We claim this is such a set as required by the claim. Indeed, let $U\< W_1$ be any set of size $r$. From the property of the graph $H$, we can find a perfect matching between $\tau^{-1}(U)\cup Y$ and $[3r]$. For each $i\in [3r]$, take the vertex $v\in U\cup W_2$ such that $\tau^{-1}(v)\in \tau^{-1}(U)\cup Y$ is matched to $i$ and also take an $(x_i,y_i)$-path of length $\ell-1$ on $V(P_i)\cup\{v\}$, i.e., in $P_i$ replacing the unique edge $uw\in P_i\cap A_v$ by the path $uvw$ (see \ref{p20}). These paths altogether cover $U\cup A$, as required.
\end{proof}

\noindent \textbf{Phase $2$. Connect almost all pairs $\{x_i,y_i\}$ for $x_i\in X_2,y_i\in Y_2$}.
 \medskip\\
Let $W'=W\setminus (A\cup W_1)$ and $s=\frac{1+c}{\ell-2}r$ for $c=\frac{1}{8}$. Then $|W'|=\frac{\ell-2}{\ell}n-3r(\ell-2)-r=(\ell-2)|X_2|-r$ and $s>2m$. By repeatedly applying Lemma~\ref{connect} with $U=W'$, $d_1=m$ and $k_i=\ell$, one can greedily connect pairs $\{x_i,y_i\}$, where $x_i\in X_2,y_i\in Y_2$, via vertex-disjoint paths $P_i$ of length $\ell-1$ with all the internal vertices in $W'$ until there are exactly $s$ pairs remaining. Since in total the number of vertices unused in the connections is at least \[|W'|-(\ell-2)(|X_2|-s)=(\ell-2)s-r=c r\ge 20m^2,\] if after this there are $t$ vertex pairs remaining, where $t>s>2m$, then take among them $2m$ pairs $\{x_i,y_i\}$. Lemma \ref{connect} would give one more path for some $\{x_i,y_i\}$, contradicting $t>s$. The process terminates, therefore, with only $s$ pairs remaining, denoted as $\mathcal{S}$.

\medskip
\noindent \textbf{Phase $3$. Connect the pairs in $\mathcal{S}$ using $W_1$}
\medskip\\
%Let $Z_1=W_3\setminus W'$ and let $S=l/(2\log n+2)$ be an integer\footnote{We assume this for a clear presentation. Where this is not an integer we make a small adjustment to $l$ so that this is true. Such a small adjustment will easily be tolerated by the proof of Lemma \ref{absorbmany}.}. Take distinct vertices $x_{i,j}\in Z_1$, $6s\log n+1\leq i\leq n/l$, $2\leq j\leq S$, and let $x_{i,1}=x_i$ and $x_{i,S+1}=y_i$, $6s\log n+1\leq i\leq n/l$. Let $Z_2\<Z_1$ be the vertices in $Z_1$ not among these labelled vertices. Connect as many pairs $x_{i,j}$ and $x_{i,j+1}$ as possible by disjoint paths of length $2\log n+2$ using vertices from $Z_2$, stopping if no more can be connected or until there are $s$ such vertex pairs remaining. If after this there are $t$ vertex pairs remaining, where $t> s$, then take among them $m$ pairs $(a_i,b_i)$, $i\in [m]$, so that the vertices $a_i$ and $b_i$ are all distinct. Let $Z_3\<Z_2$ be the set of vertices not covered by any of the paths found so far. Now $|Z_3|=(2\log n+1)t-(|W_1|+|W_2|)/2\geq s\geq 1000m$. Given any set $U\<V(G)$ with $|U|=m$, $|N(U,Z_3)|\geq |Z_3|-2m\geq (1-1/128)|Z_3|$. Therefore, the graph $G[Z_3\cup(\cup_i\{a_i,b_i\})]$ satisfies the conditions of Lemma \ref{connect}, and so, by that lemma, there is a path of length $2\log n+2$ between one of the pairs $(a_i,b_i)$ in this graph, contradicting $t>s$. The process stops, therefore, with only $s$ pairs remaining.
Let $S$ be the set of indices $i$ such that $\{x_i,y_i\}\in \mathcal{S}$ and $L$ be the set of remaining vertices in $W'$ not covered by any path in previous steps. Then $|L|=|W'|-(\ell-2)(|X_2|-s)=(\ell-2)s-r=c r$. Write $X_S=\{x_i: i\in S\},Y_S=\{y_i: i\in S\}$ and $L=\{u_1,u_2,\ldots,u_{c r}\}$.

Now we shall connect all pairs $\{x_i,y_i\}\in \mathcal{S}$ via pairwise disjoint paths of length exactly $\ell-1$, each having internal vertices in $W_1$. To achieve this, we choose a partition $T_1\cup T_2=W_1$ such that $|T_1|=3|T_2|$ and $G$ $\frac{d_1}{20}$-expands into $T_i$ for $i\in[2]$. Since $G$ $d_1$-expands into $W_1$ and $\frac{|T_i|}{5|W_1|}d_1\ge\frac{d_1}{20}>2\log 2r$, such a partition can be obtained by applying Lemma~\ref{randompart}.

As $G$ $\frac{d_1}{20}$-expands into $T_1$, we first find a star-matching $M'$ into $T_1$ from $X_S\cup Y_S\cup L$, such that $x_i$ (or $y_i$) is matched to one vertex, say $x_i^1$ (resp. $y_i^1$), and every $u_j\in L$ is matched to two vertices, say $v_j^1$ and $v_j^2$, respectively. Indeed, by Hall's condition, it suffices to guarantee that $N(X,T_1)\ge 2|X|$ for every subset $X\in X_S\cup Y_S\cup L$: the case $|X|<m$ easily follows as $G$ $\frac{d_1}{10}$-expands into $T_1$;
%\todo{This is the only point to use expands-into}
as $G$ is $m$-joined, if $|X|\ge m$, then $N(X,T_1)>|T_1|-m\ge r\ge 2|X|$ because $|X|\le |X_S\cup Y_S\cup L|=2s+c r\le \frac{r}{2}$. Write $T_1'=T_1-V(M')$. We then arbitrarily partition all the pairs $\{v_j^1,v_j^2\}, j\in[c r]$ into $s$ groups of nearly equal size, denoted by $U_1,U_2,\ldots,U_s$.

Next we shall find a maximal collection of pairwise disjoint $(x_i,y_i)$-paths $P_i$ ($i\in S$) of length $\ell-1$ with all internal vertices in $T_1$ such that each $P_i$ passes through the paths $v_j^1u_jv_j^2$ for all $\{v_j^1,v_j^2\}\in U_i$. Note that \[\frac{c r}{s}-1\le|U_i|\le \frac{c r}{s}+1=\frac{c}{1+c}(\ell-2)+1\le c(\ell-2)= \frac{\ell-2}{8}\] and thus $\ell> 3(|U_i|+1)+2|U_i|$. Setting $U_i=\{\{v_j^1,v_j^2\}: j\in[t]\}$ for some $t\in\N$ to ease the notation, we will build such a path $P_i$ by connecting the corresponding $|U_i|+1$ pairs
\[\{x_i, v_1^1\},\{v_1^2,v_2^1\},\{v_2^2,v_3^1\},\ldots,\{v_j^2,v_{j+1}^1\},\{v_{j+1}^2,v_{j+2}^1\},\ldots,\{v_t^2,y_i\}\]
of vertices via pairwise disjoint paths of nearly equal length $\ell_1,\ell_2,\ldots,\ell_{t+1}$ with $\ell_j\ge 3$ and $\sum_{j\in[t+1]}\ell_j=\ell-2t-1$. To do this, in total we have a set of $\sum_{i\in S}(|U_i|+1)$ pairs to connect as above, say $\mathcal{P}_S$, and the number of vertices in $T_1$ used in all these connections is at most $s(\ell-2)-|L|=r$. Note that $|T_1|=\frac{3r}{2}\ge r+20m^2$ and every subpath in the connections has length \[3\le\ell_j\le\frac{\ell-1-2t}{t+1}+1\le\frac{\ell}{ c r/s}\le\frac{2\ell}{c(\ell-2)}\le \frac{4}{c}=:\ell_{\max}.\] Connect as many pairs in $\mathcal{P}_S$ as possible via pairwise vertex-disjoint paths of prescribed lengths $\ell_j$. Then by repeatedly applying Lemma~\ref{connect} with $d_1=m$ and $U$ playing the role of the remaining set of at least $20m^2$ vertices in $T_1$ in the process, there is a subfamily $\mathcal{P}'\<\mathcal{P}_S$ of at most $2m$ pairs left. Now we shall connect the remaining pairs in $\mathcal{P}'$ using $T_2$. Note that $G$ $\frac{d_1}{20}$-expands into $T_2$ and thus as $\frac{1}{C}\ll\frac{1}{\ell}$, \[d(v,T_2)\ge \frac{d_1}{20}\ge\frac{|W_1|}{100|W|}d\ge \frac{d}{10^5\ell}\ge 10m\ell_{\max}.\]
Then one can greedily complete the connections for pairs in $\mathcal{P}'$ because the total number of vertices used in $T_2$ is at most $2m\ell_{\max}$.

In summary, we obtain a set $A$ (see Claim~\ref{absorbmany}) and a collection of pairwise disjoint paths $P_1,P_2,\ldots, P_{\frac{n}{\ell}}$, where the first $3r$ paths have length $\ell-2$ whilst the remaining paths have length $\ell-1$. Let $W_1'$ be the vertices of $W_1$ not used in any of the previous paths $P_i$. Then $|W_1'|=|W_1|+|L|-s(\ell-2)=2r+cr-(1+c)r=r$. Using the properties of the set $A$, we can find a collection of $3r$ pairwise disjoint paths of length $\ell-1$, each connecting $x_i$ and $y_i$ for every $i\in[3r]$ whilst covering the set $A\cup W_1'$. This completes all the desired paths which cover the entire vertex set.
\oof

\bibliographystyle{abbrv}
\bibliography{ref}

\begin{thebibliography}{10}

\bibitem{ALLEN19}
P.~Allen, J.~B\"{o}ttcher, H.~H\`{a}n, Y.~Kohayakawa, and Y.~Person.
\newblock Blow-up lemmas for sparse graphs.
\newblock {\em arXiv:1612.00622v4}, 2019.

\bibitem{A94}
N.~Alon.
\newblock Explicit ramsey graphs and orthonormal labelings.
\newblock {\em Electron. J. Combin.}, 1:Research Paper 12, approx. 8, 1994.

\bibitem{A2}
N.~Alon and V.~Asodi.
\newblock Sparse universal graphs.
\newblock volume 142, pages 1--11. 2002.
\newblock Probabilistic methods in combinatorics and combinatorial
  optimization.

\bibitem{A3}
N.~Alon and M.~Capalbo.
\newblock Sparse universal graphs for bounded-degree graphs.
\newblock {\em Random Structures Algorithms}, 31(2):123--133, 2007.

\bibitem{A4}
N.~Alon, M.~Capalbo, Y.~Kohayakawa, V.~R\"{o}dl, A.~Ruci\'{n}ski, and
  E.~Szemer\'{e}di.
\newblock Universality and tolerance (extended abstract).
\newblock In {\em 41st {A}nnual {S}ymposium on {F}oundations of {C}omputer
  {S}cience ({R}edondo {B}each, {CA}, 2000)}, pages 14--21. IEEE Comput. Soc.
  Press, Los Alamitos, CA, 2000.

\bibitem{AK98}
N.~Alon and N.~Kahale.
\newblock Approximating the independence number via the {$\theta$}-function.
\newblock {\em Math. Programming}, 80(3, Ser. A):253--264, 1998.

\bibitem{Alon07}
N.~Alon, M.~Krivelevich, and B.~Sudakov.
\newblock Embedding nearly-spanning bounded degree trees.
\newblock {\em Combinatorica}, 27(6):629--644, 2007.

\bibitem{B6}
L.~Babai, F.~R.~K. Chung, P.~Erd\H{o}s, R.~L. Graham, and J.~H. Spencer.
\newblock On graphs which contain all sparse graphs.
\newblock In {\em Theory and practice of combinatorics}, volume~60 of {\em
  North-Holland Math. Stud.}, pages 21--26. North-Holland, Amsterdam, 1982.

\bibitem{BCPS10}
J.~Balogh, B.~Csaba, M.~Pei, and W.~Samotij.
\newblock Large bounded degree trees in expanding graphs.
\newblock {\em Electron. J. Combin.}, 17(1):Research Paper 6, 9, 2010.

\bibitem{BECK}
J.~Beck.
\newblock {\em Combinatorial games}, volume 114 of {\em Encyclopedia of
  Mathematics and its Applications}.
\newblock Cambridge University Press, Cambridge, 2008.
\newblock Tic-tac-toe theory.

\bibitem{B12}
S.~N. Bhatt, F.~R.~K. Chung, F.~T. Leighton, and A.~L. Rosenberg.
\newblock Universal graphs for bounded-degree trees and planar graphs.
\newblock {\em SIAM J. Discrete Math.}, 2(2):145--155, 1989.

\bibitem{Bo98}
B.~Bollob\'{a}s.
\newblock {\em Modern graph theory}, volume 184 of {\em Graduate Texts in
  Mathematics}.
\newblock Springer-Verlag, New York, 1998.

\bibitem{han}
J.~B\"{o}ttcher, J.~Han, Y.~Kohayakawa, R.~Montgomery, O.~Parczyk, and
  Y.~Person.
\newblock Universality for bounded degree spanning trees in randomly perturbed
  graphs.
\newblock {\em Random Structures Algorithms}, 55(4):854--864, 2019.

\bibitem{C16}
M.~R. Capalbo and S.~R. Kosaraju.
\newblock Small universal graphs.
\newblock In {\em Annual {ACM} {S}ymposium on {T}heory of {C}omputing
  ({A}tlanta, {GA}, 1999)}, pages 741--749. ACM, New York, 1999.

\bibitem{C17}
F.~R.~K. Chung and R.~L. Graham.
\newblock On graphs which contain all small trees.
\newblock {\em J. Combinatorial Theory Ser. B}, 24(1):14--23, 1978.

\bibitem{C19}
F.~R.~K. Chung and R.~L. Graham.
\newblock On universal graphs for spanning trees.
\newblock {\em J. London Math. Soc. (2)}, 27(2):203--211, 1983.

\bibitem{C20}
F.~R.~K. Chung, R.~L. Graham, and N.~Pippenger.
\newblock On graphs which contain all small trees. {II}.
\newblock In {\em Combinatorics ({P}roc. {F}ifth {H}ungarian {C}olloq.,
  {K}eszthely, 1976), {V}ol. {I}}, volume~18 of {\em Colloq. Math. Soc.
  J\'{a}nos Bolyai}, pages 213--223. North-Holland, Amsterdam, 1978.

\bibitem{D12}
D.~Dellamonica, Jr., Y.~Kohayakawa, V.~R\"{o}dl, and A.~Ruci\'{n}ski.
\newblock An improved upper bound on the density of universal random graphs.
\newblock In {\em L{ATIN} 2012: theoretical informatics}, volume 7256 of {\em
  Lecture Notes in Comput. Sci.}, pages 231--242. Springer, Heidelberg, 2012.

\bibitem{FER19}
A.~Ferber, G.~Kronenberg, and K.~Luh.
\newblock Optimal threshold for a random graph to be 2-universal.
\newblock {\em Trans. Amer. Math. Soc.}, 372(6):4239--4262, 2019.

\bibitem{FN18}
A.~Ferber and R.~Nenadov.
\newblock Spanning universality in random graphs.
\newblock {\em Random Structures Algorithms}, 53(4):604--637, 2018.

\bibitem{FRANK21}
K.~Frankston, J.~Kahn, B.~Narayanan, and J.~Park.
\newblock Thresholds versus fractional expectation-thresholds.
\newblock {\em Ann. of Math. (2)}, 194(2):475--495, 2021.

\bibitem{Glebov13}
R.~Glebov, D.~Johannsen, and M.~Krivelevich.
\newblock Hitting time appearance of certain spanning trees in the random graph
  process.
\newblock in preparation.

\bibitem{Haxell01}
P.~E. Haxell.
\newblock Tree embeddings.
\newblock {\em J. Graph Theory}, 36(3):121--130, 2001.

\bibitem{HKS09}
D.~Hefetz, M.~Krivelevich, and T.~Szab\'{o}.
\newblock Hamilton cycles in highly connected and expanding graphs.
\newblock {\em Combinatorica}, 29(5):547--568, 2009.

\bibitem{JKS12}
D.~Johannsen, M.~Krivelevich, and W.~Samotij.
\newblock Expanders are universal for the class of all spanning trees (extended
  abstract).
\newblock In {\em Proceedings of the {T}wenty-{T}hird {A}nnual {ACM}-{SIAM}
  {S}ymposium on {D}iscrete {A}lgorithms}, pages 1539--1551. ACM, New York,
  2012.

\bibitem{JKV}
A.~Johansson, J.~Kahn, and V.~Vu.
\newblock Factors in random graphs.
\newblock {\em Random Structures Algorithms}, 33(1):1--28, 2008.

\bibitem{KK07}
J.~Kahn and G.~Kalai.
\newblock Thresholds and expectation thresholds.
\newblock {\em Combin. Probab. Comput.}, 16(3):495--502, 2007.

\bibitem{kahn16}
J.~Kahn, E.~Lubetzky, and N.~Wormald.
\newblock The threshold for combs in random graphs.
\newblock {\em Random Structures Algorithms}, 48(4):794--802, 2016.

\bibitem{2003}
J.~H. Kim.
\newblock Perfect matchings in random uniform hypergraphs.
\newblock {\em Random Structures Algorithms}, 23(2):111--132, 2003.

\bibitem{KL14}
J.~H. Kim and S.~J. Lee.
\newblock Universality of random graphs for graphs of maximum degree two.
\newblock {\em SIAM J. Discrete Math.}, 28(3):1467--1478, 2014.

\bibitem{KSS01}
J.~Koml\'{o}s, G.~N. S\'{a}rk\"{o}zy, and E.~Szemer\'{e}di.
\newblock Spanning trees in dense graphs.
\newblock {\em Combin. Probab. Comput.}, 10(5):397--416, 2001.

\bibitem{1997}
M.~Krivelevich.
\newblock Triangle factors in random graphs.
\newblock {\em Combin. Probab. Comput.}, 6(3):337--347, 1997.

\bibitem{Kri10}
M.~Krivelevich.
\newblock Embedding spanning trees in random graphs.
\newblock {\em SIAM J. Discrete Math.}, 24(4):1495--1500, 2010.

\bibitem{KS03}
M.~Krivelevich and B.~Sudakov.
\newblock Sparse pseudo-random graphs are {H}amiltonian.
\newblock {\em J. Graph Theory}, 42(1):17--33, 2003.

\bibitem{KS06}
M.~Krivelevich and B.~Sudakov.
\newblock Pseudo-random graphs, more sets, graphs and numbers.
\newblock {\em Bolyai Soc. Math. Stud.}, 15:199--262, 2006.

\bibitem{LPS88}
A.~Lubotzky, R.~Phillips, and P.~Sarnak.
\newblock Ramanujan graphs.
\newblock {\em Combinatorica}, 8(3):261--277, 1988.

\bibitem{Montgomery14}
R.~Montgomery.
\newblock Embedding bounded degree spanning trees in random graphs.
\newblock {\em preprint}, arXiv:1405.6559, 2014.

\bibitem{Montgomery19}
R.~Montgomery.
\newblock Spanning trees in random graphs.
\newblock {\em Adv. Math.}, 356:106793, 92, 2019.

\bibitem{PP}
J.~Park and H.~Pham.
\newblock A proof of the kahn-kalai conjecture.
\newblock {\em arXiv:2203.17207}, 2022.

\bibitem{TALA10}
M.~Talagrand.
\newblock Are many small sets explicitly small?
\newblock In {\em S{TOC}'10---{P}roceedings of the 2010 {ACM} {I}nternational
  {S}ymposium on {T}heory of {C}omputing}, pages 13--35. ACM, New York, 2010.

\end{thebibliography}

\begin{appendices}
\section{Proof of Lemma~\ref{connect}}\label{APP1}
We will also find the following result in~\cite{han}, which is preciously formulated in \cite{BCPS10}.
\begin{theorem}\cite[Corollary 6]{han}\label{almost2}
Let $\Delta,m,M\in \N$ and $H$ be a non-empty graph such that for every $X\<V(H)$, if $0<|X|\leq m$, then $N_H(X)|\geq \Delta|X|+1$ and, if $|X|=m$, then $|N_H(X)|\geq \Delta m+M$, then $H$ contains every tree in $\mathcal{T}(M,\Delta)$. Moreover, for any $T\in \mathcal{T}(M,\Delta)$ and any vertices $x\in V(T),y\in V (G)$, there exists
an embedding of $T$ into $G$ which maps $x$ to $y$.
\end{theorem}
\pr[Proof of Lemma~\ref{connect}]
Recall that $G$ is $m$-joined and contains disjoint vertex sets $X=\{x_1,\ldots,x_{2m}\}$, $Y=\{y_1,\ldots, y_{2m}\}$ and $U$ with $|U|\ge20d_1m$.
Divide $U$ into two sets, $U_1$ and $U_2$, each of size $\frac{|U|}{2}$. Pick a maximal subset $B\<U_1$ such that $|B|\leq m$ and $|N(B,U_1)|<2d_1|B|$.  Let $V_1=U_1\setminus B$.\medskip

\noindent\textbf{Claim. }
  Every $A\<V_1$ with $0<|A|\leq m$ has $|N(A,V_1)|\ge2d_1|A|$.
\pr[Proof of claim]\renewcommand*{\qedsymbol}{$\blacksquare$}
Suppose for contradiction that there is a set $A\<V_1$ with $0<|A|\leq m$ and $|N(A,V_1)|<2d_1|A|$. Then $|N(A\cup B,U_1)|<2d_1(|A|+|B|)$, so by the maximality of $B$ we must have that $m<|A\cup B|\le 2m$. As $G$ is $m$-joined, we have $|N(A\cup B, U_1)|\geq |U_1|-|A\cup B|-m$. Therefore,
\[
|N(A,V_1)|\geq |N(A\cup B,U_1)|-|N(B,U_1)|\geq |U_1|-3m-2d_1m\geq 2d_1m.
\]
\oof
We have that $|V_1|\ge |U_1|-m=\frac{|U|}{2}-m$ and every set $A\<V_1$ with $|A|\leq m$ satisfies $|N(A, V_1)|\geq2d_1|A|$.
Similarly, find a set $V_2\<U_2$ with the same expansion property, with $|V_2|\ge\frac{|U|}{2}-m$.

Now as $G$ is $m$-joined, if $X'\<X$ is a set of size $m$, then $|N(X',V_1)|>|V_1\setminus X'|-m$ and hence some vertex $x\in X'$ must have at least $\frac{|V_1|-2m}{m}\geq 2d_1$ neighbours in the graph $V_1$. Therefore, at least $m+1$ vertices in $X$ have at least $2d_1$ neighbours in $V_1$.
Similarly at least $m+1$ vertices in $Y$ have at least $2d_1$ neighbours in $V_2$. Therefore, there is some index $j\in[2m]$ for which $x_j$ and $y_j$ have at least $2d_1$ neighbours in $V_1$ and $V_2$, respectively.

The graph $H=G[V_1\cup\{x_j\}]$ then has the property that, given any set $A\<V(H)$, if~$0<|A|\leq m$, then $|N_H(A)|\geq d_1|A|+1$, and, if $|A|=m$, then $|N_H(A)|\ge |H|-2m$. Let $T$ be the $d$-ary tree of height $\ell=\lceil \log m/\log d_1\rceil$. As $k_j\geq 2\log m/\log d_1+1$, we have that $\lfloor k_j/2\rfloor-\ell-1\geq 0$. Attach a path of length $\lfloor k_j/2\rfloor-\ell-1$ to the root of $T$ to get the tree $T'$ and let the end vertex of the path which is not the root of $T$ be $t_1$. The tree $T'$ has at most $k_j/2+2d_1^{\ell}\le \frac{|U|}{4}+2m(d_1+1)< |H|-2m-2md_1$ vertices. Therefore, by Theorem \ref{almost2}, $H$ contains a copy of $T'$ so that the vertex $t_1$ is embedded onto the vertex $x_j$. Say this copy of $T'$ is $S_1$.
Similarly, $G[V_2\cup\{y_j\}]$ contains a $d_1$-ary tree with a path of length $\lceil k_j/2 \rceil -\ell$ connecting the root of the regular tree to $y_j$. Call this tree $S_2$.

The set of vertices in the last level of the $d_1$-ary trees each contain at least $m$ vertices. Suppose these sets are $W_1$ and $W_2$ for the trees $S_1$ and $S_2$ respectively. Then there exists an edge $v_1v_2$ with $v_1\in W_1,v_2\in W_2$. Taking the path of length $\lceil k_j/2\rceil-1$ through the tree $S_1$ from $x_j$ to $v_1$, the path $v_1v_2$ and the path of length $\lfloor k_j/2\rfloor$ through the tree $S_2$ from $v_2$ to $y_j$, we obtain a desired $(x_j,y_j)$-path of length $k_j$, with internal vertices in $U$.
\oof
\section{Proof of Proposition~\ref{f1}}\label{APP2}

\pr
Given constants $p,\be>0$ and an $n$-vertex $(p,\beta)$-bijumbled graph $G$ with minimum degree $\de\ge4\sqrt{p\be n}$, we prove that $G$ is an $(n, d_1)$-expander with $d_1=\frac{pn}{4\beta}\ge 100$. Indeed, it is easy to verify that $G$ is $\frac{n}{2d_1}$-joined since for every two disjoint vertex sets $X,Y$ size at least $\frac{n}{2d_1}=\frac{2\be}{p}$ we have \[e(X,Y)\ge p|X||Y|-\be\sqrt{|X||Y|}\ge \frac{4\be^2}{p}-\frac{2\be^2}{p}>0.\] By definition~\ref{expander}, it remains to verify that for every $X\<V(G)$ with $|X|<\frac{n}{2d_1}$, $N(X)\ge d_1|X|$, and we split it into two cases, where we write $Y=N(X)$.

 If $|X|\ge 100\frac{\be^2}{p^2n}$, then as $|V\setminus X|\ge \frac{3n}{4}$ and \[p|X||Y|+\be\sqrt{|X||Y|}\ge e(X,N(X))=e(X,V\setminus X)\ge p|X||V\setminus X|-\be\sqrt{|X||V\setminus X|},\] we obtain that $p|Y|+\be\sqrt{\frac{|Y|}{|X|}}\ge \frac{3}{4}pn-\be\sqrt{\frac{n}{|X|}}$. This implies that \[|Y|\ge \frac{3}{4}n-\frac{2\be}{p}\sqrt{\frac{n}{|X|}}\ge \frac{3}{4}n-\frac{2\be}{p}\sqrt{\frac{p^2n^2}{100\be^2}}>\frac{n}{2}\ge d_1|X|.\]

 If $|X|< 100\frac{\be^2}{p^2n}$ and assume for contradiction that $|Y|<d_1|X|$, then $\de |X|\le e(X,X\cup Y)\le p|X||X\cup Y|+\be\sqrt{|X||X\cup Y|}\le p|X|^2(d_1+1)+\be|X|\sqrt{d_1+1}$. Recall that $d_1=\frac{pn}{4\beta}\ge 100$ and $|X|<100\frac{\be^2}{p^2n}$. It follows that \[\de\le p(d_1+1)|X|+\be\sqrt{d_1+1}< 50\be+\sqrt{\frac{p\be n}{2}}\le50\sqrt{\frac{p\be n}{400}}+\sqrt{\frac{p\be n}{2}}<4\sqrt{p\be n}, \] a contradiction.

The case when $G$ is an $(n, d, \la)$-graph is much easier. First reset $d_1=\frac{d}{2\la}>4$ and by Expander Mixing Lemma, $e(X,Y)\ge \frac{d}{n}|X||Y|-\la\sqrt{|X||Y|(1-\frac{|X|}{n})(1-\frac{|Y|}{n})}> \frac{d}{n}(\frac{\la n}{d})^2-\frac{\la^2 n}{d}=0$ for every two disjoint vertex sets $X,Y$ size at least $\frac{n}{2d_1}=\frac{\la n}{d}$. Then we will show that for every $X\<V(G)$ with $|X|<\frac{\la n}{d}$, $N(X)\ge d_1|X|$. Suppose this is invalid for some $X\<V(G)$, i.e, $N(X)< d_1|X|$. Similarly set $Y=N(X)$ and thus \[d |X|\le e(X,X\cup Y)< \frac{d}{n}|X||X\cup Y|+\la\sqrt{|X||X\cup Y|}\le \frac{d}{n}(d_1+1)|X|^2+\la|X|\sqrt{d_1+1}.\] Then it implies that $d<\frac{d}{n}(d_1+1)|X|+\la \sqrt{d_1+1}\le \la(\frac{d}{2\la}+1)+\la \sqrt{\frac{d}{2\la}+1}$, i.e, $\frac{d}{2\la}<\sqrt{\frac{d}{2\la}+1}+1$, a contradiction to the fact that $\la<\frac{d}{8}$.
\oof
\end{appendices}
\end{document}